\documentclass[12pt,a4paper,reqno,twoside]{amsart}

\usepackage[english]{babel}
\usepackage{stmaryrd}
\usepackage{dsfont}
\usepackage[symbol*,ragged]{footmisc}
\usepackage[colorlinks,linkcolor=red,anchorcolor=blue,citecolor=blue,urlcolor=blue]{hyperref}
\usepackage{color,xcolor}

\usepackage{geometry}
\usepackage{amssymb}
\usepackage{amsmath}
\usepackage{mathrsfs}
\usepackage{amsfonts}
\usepackage{epsfig}

\usepackage{amsthm}
\usepackage{amsxtra}
\usepackage{bbding}
\usepackage{epsfig}
\usepackage{graphicx}
\usepackage{latexsym}
\usepackage{mathbbol}
\usepackage{bbold}

\usepackage{pifont}
\usepackage{wasysym}
\usepackage{skull}
\usepackage{float}

\DeclareSymbolFontAlphabet{\mathbb}{AMSb}
\DeclareSymbolFontAlphabet{\mathbbol}{bbold}

\usepackage{amscd}
\usepackage[all]{xy}
\allowdisplaybreaks[4]
\usepackage{setspace}

\theoremstyle{plain}
\newtheorem{theorem}{\normalfont\scshape Theorem}[section]
\newtheorem{proposition}{\normalfont\scshape Proposition}[section]
\newtheorem{lemma}[proposition]{\normalfont\scshape Lemma}

\newtheorem*{corollary*}{\normalfont\scshape Corollary}

\theoremstyle{remark}
\newtheorem*{remark*}{\normalfont\scshape Remark}
\newtheorem*{notation}{\normalfont\scshape Notation}

\numberwithin{equation}{section}
\addtocounter{footnote}{1}

\renewcommand{\footnoterule}{
  \kern -3pt
  \hrule width 2.5in height 0.4pt
  \kern 3pt
}

\makeatletter
\@ifundefined{MakeUppercase}{}{}
\makeatother

\begin{document}
	
\title[ Almost primes of the form ${[}p^{1/\gamma}{]}$ ]
	  { Almost primes of the form $[p^{1/\gamma}]$ }

\author[Fei Xue, Jinjiang Li, Min Zhang]
       {Fei Xue \quad \& \quad Jinjiang Li \quad \& \quad Min Zhang}

\address{Department of Mathematics, China University of Mining and Technology,
         Beijing 100083, People's Republic of China}

\email{fei.xue.math@gmail.com}

\address{(Corresponding author) Department of Mathematics, China University of Mining and Technology,
         Beijing 100083, People's Republic of China}

\email{jinjiang.li.math@gmail.com}

\address{School of Applied Science, Beijing Information Science and Technology University,
		 Beijing 100192, People's Republic of China  }

\email{min.zhang.math@gmail.com}

\date{}

\footnotetext[1]{Jinjiang Li is the corresponding author. 

\quad\,\,
{\textbf{MR(2020) Subject Classification}}: 11A41, 11N80, 11N36, 11L07

}

\begin{abstract}
Let $\mathcal{P}_r$ denote an almost--prime with at most $r$ prime factors, counted according to multiplicity.
In this paper, it is proved that, for $0.989<\gamma<1$, there exist infinitely many primes $p$ such that $[p^{1/\gamma}]=\mathcal{P}_7$, which constitutes an improvement upon the previous result of
Banks--Guo--Shparlinski \cite{Banks-Guo-Shparlinski-2016} who showed that there exist infinitely many primes $p$ such that $[p^{1/\gamma}]=\mathcal{P}_8$ for $\gamma$ near to one.

\end{abstract}

\maketitle

\section{Introduction and main result}
Let $\gamma\in(\frac{1}{2},1)$ be a fixed real number. The Piatetski--Shapiro sequences are sequences of the form
\begin{equation*}
 \mathscr{N}_{\gamma}:=\big\{[n^{1/\gamma}]:\,n\in \mathbb{N}^+\big\}.
\end{equation*}
Such sequences have been named in honor of Piatetski--Shapiro, who \cite{Piatetski-Shapiro-1953}, in
1953, proved that $\mathscr{N}_{\gamma}$ contains infinitely many primes provided that $\gamma\in(\frac{11}{12},1)$. The prime numbers of the form $p=[n^{1/\gamma}]$ are called \textit{Piatetski--Shapiro primes of type $\gamma$}. More precisely, for such $\gamma$ Piatetski--Shapiro \cite{Piatetski-Shapiro-1953} showed that the counting function
\begin{equation*}
 \pi_\gamma(x):=\#\big\{\textrm{prime}\,\, p\leqslant x:\,p=[n^{1/\gamma}]\,\,\textrm{for some}\,\,
 n\in\mathbb{N}^+ \big\}
\end{equation*}
satisfies the asymptotic property
\begin{equation*}
\pi_{\gamma}(x)=\frac{x^{\gamma}}{\log x}(1+o(1))
\end{equation*}
as $x\to\infty$. Since then, the range for $\gamma$ of the above asymptotic formula in which it is known that $\mathscr{N}_{\gamma}$ contains infinitely many primes has been enlarged many times (see the literatures \cite{Kolesnik-1967,Leitmann-1975,Leitmann-1980,Heath-Brown-1983,Kolesnik-1985,Liu-Rivat-1992,
Rivat-1992,Rivat-Sargos-2001}) over the years and is currently known to hold for all $\gamma\in(\frac{2426}{2817},1)$ thanks to Rivat and Sargos \cite{Rivat-Sargos-2001}. Rivat and Wu \cite{Rivat-Wu-2001} also showed that there exist infinitely many Piatetski--Shapiro primes for $\gamma\in(\frac{205}{243},1)$ by showing a lower bound of $\pi_\gamma(x)$ with the expected order of magnitude. We remark that if $\gamma>1$ then $\mathscr{N}_\gamma$ contains all natural numbers, and hence all primes, particularly. On the other hand, Deshouillers \cite{Deshouillers-1976} proved that $\pi_{\gamma}(x)\to\infty$ for almost all positive $\gamma\in(0,1)$ (in the sense of Lebesgue measure). However, this result provides no specific value of $\gamma$. Actually, the proof of Deshouillers \cite{Deshouillers-1976} gave a little more, i.e., for almost all positive $\gamma$
there holds
\begin{equation*}
\limsup_{x\to\infty}\frac{\pi_{\gamma}(x)}{x^\gamma/(\log x)}\geqslant1.
\end{equation*}
Many authors have studied arithmetic properties of Piatetski--Shapiro sequences (see Baker
et al. \cite{Baker-Banks-Brudern-Shparlinski-Weingartner-2013} and the references contained therein), and it is natural to ask whether certain properties also hold on special subsequences of the Piatetski--Shapiro sequences. Perhaps the most important of these are the subsequences of the form
\begin{equation*}
 \mathscr{M}_{\gamma}:=\big\{[p^{1/\gamma}]:\,\,p\,\,\textrm{is prime number} \big\}.
\end{equation*}
It is conjectured that the counting function
\begin{equation*}
 \Pi_\gamma(x):=\#\big\{\textrm{prime}\,\, p\leqslant x:\,p=[q^{1/\gamma}],\,\,\textrm{$q$ is prime} \big\}
\end{equation*}
satisfies the asymptotic property
\begin{equation}\label{double-PS-conjecture}
\Pi_{\gamma}(x)=\frac{x^{\gamma}}{\gamma\log^2 x}(1+o(1))
\end{equation}
as $x\to\infty$.  This is not trivial even for $\gamma>1$. Indeed, the problem of representing primes by linear
polynomials in prime variable is still not solved. One can therefore look on the investigation of (\ref{double-PS-conjecture}) for $\gamma>1$ as a progress towards the twin prime conjecture.

In 1987, Balog \cite{Balog-1987} essentially showed that for almost all $\gamma\in(0,1)$, the counting function
$\Pi_\gamma(x)$ satisfies
\begin{equation*}
\limsup_{x\to\infty}\frac{\Pi_{\gamma}(x)}{x^\gamma/(\gamma\log^2 x)}\geqslant1,
\end{equation*}
but this result gives no information for any specific choice of $\gamma$. In the same paper, Balog \cite{Balog-1987} also showed that
\begin{equation*}
\sum_{\substack{\mathcal{P}_s\leqslant x\\ [\mathcal{P}_r^{1/\gamma}]=\mathcal{P}_s}}1\gg\frac{x^\gamma}{\log^2x},
\end{equation*}
where $\{r,s\}$ satisfies the following table
\begin{center}
\begin{tabular}{|c|c|c|c|c|}
		%\hline
		%\vphantom{\Big|}$R$  &  $c_R$ & $R$  &  $c_R$ & $R$  &  $c_R$ \\
		\hline
		$r$ & $1$ & $9$ & $2$ & $5$  \\
        \hline
		$s$ & $9$ & $1$ & $5$ & $2$  \\
		\hline
	\end{tabular}
\end{center}
In 2016, Banks--Guo--Shparlinski \cite{Banks-Guo-Shparlinski-2016} generalized and enhanced the result of Balog \cite{Balog-1987}. They showed that for every $\gamma\in(0,1)\setminus\{1/n:\,n\in\mathbb{Z}^+\}$ there exist infinitely many members of $[p^{1/\gamma}]$ having at most $R(\gamma)$ prime factors, giving explicit estimates for $R(\gamma)$ when $\gamma$ is near one. To be specific, for each $R=8,9,\dots,19$, there exists explicit value $\gamma_R$ such that
\begin{equation*}
\big\{[p^{1/\gamma}]\leqslant x:\,\,\textrm{$p$ is prime},\,\,[p^{1/\gamma}]
  =\mathcal{P}_R\big\}\gg\frac{x^\gamma}{\log^2x}
\end{equation*}
holds for all sufficiently large $x$ provided that $\gamma\in(\gamma_R,1)$. When $\gamma$ is near one,  Banks--Guo--Shparlinski \cite{Banks-Guo-Shparlinski-2016} showed that $R=8$ and $\gamma_8=1/1.0521=0.950479\dots$.

In this paper, we shall improve the result of Banks--Guo--Shparlinski
\cite{Banks-Guo-Shparlinski-2016} and establish the following theorem.

\begin{theorem}\label{Theorem-1}
For $0.989<\gamma<1$, there exist infinitely many primes $p$ such that $[p^{1/\gamma}]=\mathcal{P}_7$. To be specific, the following estimates
\begin{equation}\label{Thm-ineq}
  \big\{[p^{1/\gamma}]\leqslant x:\,\,\textrm{$p$ is prime},\,\,[p^{1/\gamma}]
  =\mathcal{P}_7\big\}\gg\frac{x^\gamma}{\log^2x}
\end{equation}
holds for all sufficiently large $x$ provided that $\gamma\in(\gamma_7,1)$ with $\gamma_7=0.989$.
\end{theorem}

\begin{remark*}
In order to compare our result with the result of Banks--Guo--Shparlinski
\cite{Banks-Guo-Shparlinski-2016}, we emphasize that the method in \cite{Banks-Guo-Shparlinski-2016} only gives
$R=8$ for $\gamma$ near to one, while our result reduce $R$ to $7$ when $\gamma$ near to one.
\end{remark*}

\begin{notation}
Throughout this paper, $x$ is a sufficiently large number; $\varepsilon$ and $\eta$ are sufficiently small positive numbers, which may be different in each occurrences. Let $p$, with or without subscripts, always denote a prime number. We use $[x],\,\{x\}$ and $\|x\|$ to denote the integral part of $x$, the fractional part of $x$ and the distance from $x$ to the nearest integer, respectively.
Denote by $\mathcal{P}_r$ an almost--prime with at most $r$ prime factors, counted according to multiplicity. As usual, $\Lambda(n),\mu(n)$ and $\Omega(n)$ denote von Mangoldt's function, M\"{o}bius' function, and the number of total prime factors of $n$, respectively. We write $\mathscr{L}=\log x$; $e(t)=\exp(2\pi it)$; $\psi(t)=t-[t]-\frac{1}{2}$. The notation $n\sim X$ means that $n$ runs through a subinterval of $(X/2,X]$, whose endpoints are not necessarily the same in the different occurrences and may depend on the outer summation variables. $f(x)\ll g(x)$ means that $f(x)=O(g(x))$; $f(x)\asymp g(x)$ means that $f(x)\ll g(x)\ll f(x)$.
\end{notation}

%\vskip 6mm
% \vspace*{-4mm}

\section{Preliminary Lemmas}
Let $\mathscr{A}$ denote a finite set of integers, $\mathscr{P}$ an infinite set of primes,
$\overline{\mathscr{P}}$ the set of primes with its elements not belonging to $\mathscr{P}$. For given $z\geqslant2$, we set
\begin{equation*}
P(z)=\prod_{\substack{p<z\\ p\in\mathscr{P}}}p.
\end{equation*}
Define the sifting function as
\begin{equation*}
S(\mathscr{A},\mathscr{P},z)=\big|\big\{a\in\mathscr{A}:(a,P(z))=1\big\}\big|.
\end{equation*}
For $d|P(z)$, define $\mathscr{A}_d=\{a\in\mathscr{A}:a\equiv0\!\pmod d\}$.
Moreover, we assume that $|\mathscr{A}_d|$ may be written in the form
\begin{equation}\label{sieve-condi-1}
|\mathscr{A}_d|=\frac{\omega(d)}{d}X+r_d,\qquad \mu(d)\not=0,\qquad (d,\overline{\mathscr{P}})=1,
\end{equation}
where $\omega(d)$ is a multiplicative function such that $0\leqslant\omega(p)<p$, $X$ is a positive number
independent of $d$, and $r_d$ is an error term which is small on average, allowing $X$
to approximate the cardinality of $\mathscr{A}$. Also, we assume that the function $\omega(p)$ is constant on average over $p$ in $\mathscr{P}$, which means that
\begin{equation}\label{sieve-condi-2}
\sum_{\substack{z_1\leqslant p<z_2\\ p\in\mathscr{P}}}\bigg(1-\frac{\omega(p)}{p}\bigg)^{-1}
\leqslant\frac{\log z_2}{\log z_1}\bigg(1+\frac{\mathcal{K}}{\log z_1}\bigg)
\end{equation}
holds for all $z_2>z_1\geqslant2$, where $\mathcal{K}$ is a constant satisfying $\mathcal{K}\geqslant1$. For details of (\ref{sieve-condi-1}) and (\ref{sieve-condi-2}), one can see the arguments (4.12)--(4.15) on page 28 of Halberstam and Richert \cite{Halberstam-Richert-book}, and the arguments on page 205 of Iwaniec \cite{Iwaniec-1981}.
\begin{lemma}\label{upper-lower-sieve}
Suppose that the conditions (\ref{sieve-condi-1}) and (\ref{sieve-condi-2}) hold. Then we have
\begin{align}
 S(\mathscr{A},\mathscr{P},z) \geqslant & \,\, XV(z)\big(f(s)+O\big(\log^{-1/3}D\big)\big)-R_D, \label{lower-sieve}
               \\
 S(\mathscr{A},\mathscr{P},z) \leqslant & \,\, XV(z)\big(F(s)+O\big(\log^{-1/3}D\big)\big)+R_D, \label{upper-sieve}
\end{align}
where
\begin{equation*}
R_D=\sum_{\substack{d<D\\ d|P(z)}}|r_d|,\qquad s=\frac{\log D}{\log z},
\end{equation*}
\begin{equation}\label{V(z)-def}
V(z)=\mathcal{C}(\omega)\frac{e^{-C_0}}{\log z}\bigg(1+O\bigg(\frac{1}{\log z}\bigg)\bigg),
\end{equation}
\begin{equation}\label{C(omega)-def}
\mathcal{C}(\omega)=\prod_p\bigg(1-\frac{\omega(p)}{p}\bigg)\bigg(1-\frac{1}{p}\bigg)^{-1},
\end{equation}
where $C_0$ denotes the Euler's constant, $f(s)$ and $F(s)$ denote the classical functions in the linear sieve
theory, which are determined by the following differential--difference equation
\begin{equation}
\begin{cases}\label{diff-eq}
F(s)=\displaystyle\frac{2e^{C_0}}{s},\quad f(s)=0, \quad 0<s\leqslant2,\\
\displaystyle\frac{\mathrm{d}}{\mathrm{d}s}(sF(s))=f(s-1),\quad \displaystyle\frac{\mathrm{d}}{\mathrm{d}s}(sf(s))=F(s-1),\quad s\geqslant2.
\end{cases}
\end{equation}
\end{lemma}
\begin{proof}
For (\ref{lower-sieve}) and (\ref{upper-sieve}), one can refer to (6), (7), (8) on page 209 of Iwaniec \cite{Iwaniec-1981}, while (\ref{diff-eq}) can be referred to as a special case with $\varkappa=1,\beta=2$ in
(9) of Iwaniec \cite{Iwaniec-1981}. Moreover, for (\ref{V(z)-def}) and (\ref{C(omega)-def}) one can see (2.4)
and (2.5) of Chapter $5$ in Halberstam and Richert \cite{Halberstam-Richert-book}.
\end{proof}

\begin{lemma}\label{psi-expansion}
For any $H>1$, one has
\begin{equation}\label{psi-expan}
\psi(t)=-\sum_{0<|h|\leqslant H}\frac{e(th)}{2\pi i h}+O(g(t,H)),
\end{equation}
where
\begin{equation*}
g(t,H):=\min\bigg(1,\frac{1}{H\|t\|}\bigg)=\sum_{h=-\infty}^{\infty}b_he(th),
\end{equation*}
and
\begin{equation*}
b_h\ll\min\bigg(\frac{\log2H}{H},\frac{1}{|h|},\frac{H}{|h|^2}\bigg).
\end{equation*}
\end{lemma}
\begin{proof}
See the arguments on page 245 of Heath--Brown \cite{Heath-Brown-1983}.
\end{proof}

\begin{lemma}\label{expo-pair-gernal}
Suppose that $f(x):[a,b]\to\mathbb{R}$ has continuous derivatives of arbitrary order on $[a,b]$, where $1\leqslant a<b\leqslant2a$. Suppose further that
\begin{equation*}
 \big|f^{(j)}(x)\big|\asymp \lambda_1 a^{1-j},\qquad j\geqslant1, \qquad x\in[a,b].
\end{equation*}
Then for any exponential pair $(\kappa,\ell)$, we have
\begin{equation*}
 \sum_{a<n\leqslant b}e(f(n))\ll \lambda_1^\kappa a^\ell+\lambda_1^{-1}.
\end{equation*}
\end{lemma}
\begin{proof}
 See (3.3.4) of Graham and Kolesnik \cite{Graham-Kolesnik-book}.
\end{proof}

\begin{lemma}\label{latticepoints}
For $\frac{1}{2}<\gamma<1,J\geqslant1,L\geqslant1,D\geqslant1,\Delta>0$, let $\mathscr{N}(\Delta)$ denote the number of solutions of
the following inequality
\begin{equation*}
  \bigg|\frac{h_1\ell_1^{1/\gamma}}{d_1} -\frac{h_2\ell_2^{1/\gamma}}{d_2}\bigg|<\Delta, \qquad
  h_1,h_2\sim J, \quad  \ell_1,\ell_2\sim L,\quad  d_1,d_2\sim D.
\end{equation*}
Then we have
\begin{equation*}
    \mathscr{N}(\Delta)\ll(JD)^\varepsilon\big(JDL+\Delta D^3JL^{2-1/\gamma}\big).
\end{equation*}
\end{lemma}
\begin{proof}
Let $h_1d_2=u_1,h_2d_1=u_2$. Then one has
\begin{align*}
            \mathscr{N}(\Delta)
\ll & \,\, (JD)^\varepsilon\cdot\#\bigg\{(u_1,u_2,\ell_1,\ell_2):\frac{JD}{4}<u_i\leqslant JD,
            \frac{L}{2}<\ell_i\leqslant L,\Big|u_1\ell_1^{1/\gamma}-u_2\ell_2^{1/\gamma}\Big|<\Delta D^2\bigg\}
                    \nonumber \\
\ll & \,\, (JD)^\varepsilon\cdot\#\bigg\{(u_1,u_2,\ell_1,\ell_2):\frac{JD}{4}<u_i\leqslant JD,\quad
            \frac{L}{2}<\ell_i\leqslant L,
                    \nonumber \\
    & \,\,\qquad\qquad\qquad  \bigg|\frac{u_1}{u_2}-\bigg(\frac{\ell_2}{\ell_1}\bigg)^{1/\gamma}\bigg|
               <16\Delta DJ^{-1}L^{-1/\gamma}\bigg\}.
\end{align*}
By using Lemma 1 of Fouvry and Iwaniec \cite{Fouvry-Iwaniec-1989} with parameters
\begin{equation*}
(\alpha,\beta,M,N,\Delta)=\big(1,\gamma^{-1},JD,L,\Delta DJ^{-1}L^{-1/\gamma}\big),
\end{equation*}
one derives that
\begin{equation*}
    \mathscr{N}(\Delta)\ll(JD)^\varepsilon\big(JDL+\Delta D^3JL^{2-1/\gamma}\big),
\end{equation*}
which completes the proof of Lemma \ref{latticepoints}.
\end{proof}

\begin{lemma}\label{Robert-Sargos-lemma}
Assume that $H\geqslant1,M\geqslant1,N\geqslant1,X>0$, and $\alpha,\beta,\gamma$ are fixed real numbers such that
$\alpha(\alpha-1)\beta\gamma\neq0$. Let $S$ be exponential sum defined as
\begin{equation*}
S=\sum_{h\sim H}\sum_{n\sim N}\Bigg|\sum_{m\sim M}
e\bigg(X\frac{m^{\alpha}h^{\beta}n^{\gamma}}{M^{\alpha}H^{\beta}N^{\gamma}}\bigg)\Bigg|.
\end{equation*}
for which the following inequality holds
\begin{equation*}
S\ll_{\varepsilon}(HNM)^{1+\varepsilon}\Bigg(\bigg(\frac{X}{HNM^2}\bigg)^{1/4}+\frac{1}{M^{1/2}}+\frac{1}{X}\Bigg).
\end{equation*}
\end{lemma}
\begin{proof}
See Theorem 3 of Robert and Sargos \cite{Robert-Sargos-2006}.
\end{proof}

\section{Preliminaries of Transformation of the Problem}
Let
\begin{equation*}
 \mathscr{A}=\big\{a:a\leqslant x, \,\, a=[p^{1/\gamma}]\big\}.
\end{equation*}
By the elementary identity
\begin{equation*}
    [-k^\gamma]-[-(k+1)^{\gamma}] =
      \begin{cases}
        1, & \textrm{if \,\,$k=[n^{1/\gamma}]$},\\
        0, & \textrm{otherwise},
      \end{cases}
\end{equation*}
we obtain
\begin{align}\label{A_d-asymp}
         \#\mathscr{A}_d
= & \,\, \#\big\{a:a\leqslant x, \,\, a=[p^{1/\gamma}],\,\, a\equiv0\!\!\!\!\pmod d\big\}
                    \nonumber \\
= & \,\, \#\bigg\{k:\frac{p^{1/\gamma}-1}{d}<k\leqslant\frac{p^{1/\gamma}}{d},\,\,p\leqslant x^\gamma\bigg\}
         =\sum_{p\leqslant x^\gamma}
         \bigg(\bigg[\frac{p^{1/\gamma}}{d}\bigg]-\bigg[\frac{p^{1/\gamma}-1}{d}\bigg]\bigg)
                    \nonumber \\
= & \,\, \frac{1}{d}\sum_{p\leqslant x^\gamma}1+\sum_{p\leqslant x^\gamma}
         \bigg(\bigg\{\frac{p^{1/\gamma}-1}{d}\bigg\}-\bigg\{\frac{p^{1/\gamma}}{d}\bigg\}\bigg)
                    \nonumber \\
= & \,\, \frac{1}{d}\pi(x^\gamma)+\sum_{p\leqslant x^\gamma}
         \bigg(\psi\bigg(\frac{p^{1/\gamma}-1}{d}\bigg)-\psi\bigg(\frac{p^{1/\gamma}}{d}\bigg)\bigg)
         =:\frac{\omega(d)}{d}\pi(x^\gamma)+R_d,
\end{align}
say. In order to use upper bound sieve and lower bound sieve, i.e., Lemma \ref{upper-lower-sieve}, it is sufficient to establish that, for some $\xi=\xi(\gamma)$, there holds the following mean value theorem
\begin{equation*}
\sum_{d\leqslant x^\xi}\bigg|\sum_{p\leqslant x^\gamma}
\bigg(\psi\bigg(\frac{p^{1/\gamma}-1}{d}\bigg)-\psi\bigg(\frac{p^{1/\gamma}}{d}\bigg)\bigg)\bigg|
\ll \frac{x^{\gamma}}{(\log x)^A},
\end{equation*}
which is equivalent to
\begin{equation}\label{mean-value-equilavent}
\sum_{d\leqslant x^\xi}\bigg|\sum_{n\leqslant x^\gamma}\Lambda(n)
\bigg(\psi\bigg(\frac{n^{1/\gamma}-1}{d}\bigg)-\psi\bigg(\frac{n^{1/\gamma}}{d}\bigg)\bigg)\bigg|
\ll \frac{x^{\gamma}}{(\log x)^A}.
\end{equation}
For convenience, we put $D=x^\xi$. Trivially, (\ref{mean-value-equilavent}) will follow, if one can prove that
for any $X\leqslant x^\gamma$, there holds
\begin{equation}\label{mean-value-split}
\sum_{d\leqslant D}\bigg|\sum_{n\sim X}\Lambda(n)
\bigg(\psi\bigg(\frac{n^{1/\gamma}-1}{d}\bigg)-\psi\bigg(\frac{n^{1/\gamma}}{d}\bigg)\bigg)\bigg|
\ll x^{\gamma}\mathscr{L}^{-A}.
\end{equation}
Let $\eta>0$ be a sufficiently small number. If $X\leqslant x^{\gamma(1-\eta)}$, then the left--hand side of
(\ref{mean-value-split}) is
\begin{align}\label{1-trivial}
\ll & \,\, \sum_{d\leqslant D}\bigg|\sum_{n\sim X}\frac{\Lambda(n)}{d}\bigg|+
           \sum_{d\leqslant D}\bigg|\sum_{n\sim X}
           \Lambda(n)\bigg(\bigg[\frac{n^{1/\gamma}}{d}\bigg]-\bigg[\frac{n^{1/\gamma}-1}{d}\bigg]\bigg)\bigg|
                     \nonumber \\
\ll & \,\, X\mathscr{L}+\mathscr{L}\sum_{d\leqslant D}\sum_{n\sim X}
           \bigg(\bigg[\frac{n^{1/\gamma}}{d}\bigg]-\bigg[\frac{n^{1/\gamma}-1}{d}\bigg]\bigg).
\end{align}
For the sum on the right--hand side of (\ref{1-trivial}), by Lemma \ref{psi-expansion} with $H=H_1:=x^{\xi+\eta}$ in (\ref{psi-expan}), we have
\begin{align*}
  & \,\, \sum_{d\leqslant D}\sum_{n\sim X}
           \bigg(\bigg[\frac{n^{1/\gamma}}{d}\bigg]-\bigg[\frac{n^{1/\gamma}-1}{d}\bigg]\bigg)
                 \nonumber \\
= & \,\, \sum_{d\leqslant D}\sum_{n\sim X}\bigg(\frac{1}{d}+
           \psi\bigg(\frac{n^{1/\gamma}-1}{d}\bigg)-\psi\bigg(\frac{n^{1/\gamma}}{d}\bigg)\bigg)
                 \nonumber \\
= & \,\, \sum_{d\leqslant D}\sum_{n\sim X}\Bigg(\frac{1}{d}+\sum_{0<|h|\leqslant H_1}\frac{1}{2\pi ih}
         \bigg(e\bigg(\frac{hn^{1/\gamma}}{d}\bigg)-e\bigg(\frac{h(n^{1/\gamma}-1)}{d}\bigg)\bigg)
                 \nonumber \\
  & \,\, \qquad+O\bigg(g\bigg(\frac{n^{1/\gamma}}{d},H_1\bigg)\bigg)
         +O\bigg(g\bigg(\frac{n^{1/\gamma}-1}{d},H_1\bigg)\bigg)\Bigg)=:S_0+S_1+S_2+S_3,
\end{align*}
say. Trivially, $S_0\ll X\mathscr{L}\ll x^{\gamma-\eta}$. For $S_2$ and $S_3$, by Lemma \ref{expo-pair-gernal}
with $(\kappa,\ell)=(\frac{1}{2},\frac{1}{2})$, we have
\begin{align*}
            S_2,S_3
 \ll & \,\, \sum_{d\leqslant D}\sum_{h=-\infty}^\infty |b_h|\bigg|\sum_{n\sim X}
            e\bigg(\frac{hn^{1/\gamma}}{d}\bigg)\bigg|
                   \nonumber \\
 \ll & \,\, |b_0|DX+\sum_{d\leqslant D}\sum_{\substack{h=-\infty\\ h\not=0}}^\infty|b_h|
            \Big(d|h|^{-1}X^{1-1/\gamma}+d^{-1/2}|h|^{1/2}X^{1/(2\gamma)}\Big)
                   \nonumber \\
 \ll & \,\, \mathscr{L}DXH_1^{-1}
            +\sum_{d\leqslant D}\sum_{0<|h|\leqslant H_1}|h|^{-1}
             \Big(d|h|^{-1}X^{1-1/\gamma}+d^{-1/2}|h|^{1/2}X^{1/(2\gamma)}\Big)
                   \nonumber \\
     & \,\, +\sum_{d\leqslant D}\sum_{|h|>H_1}H_1|h|^{-2}
             \Big(d|h|^{-1}X^{1-1/\gamma}+d^{-1/2}|h|^{1/2}X^{1/(2\gamma)}\Big)
                   \nonumber \\
 \ll & \,\, \mathscr{L}DXH_1^{-1}+D^2X^{1-1/\gamma}+D^{1/2}H_1^{1/2}X^{1/(2\gamma)}
                   \nonumber \\
 \ll & \,\, x^{\gamma-\eta}+x^{2\xi}X^{1-1/\gamma}+x^{\xi+\eta/2}\cdot x^{(1-\eta)/2}\ll x^\gamma\mathscr{L}^{-A},
\end{align*}
provided that
\begin{equation*}
\gamma>\frac{1}{2}+\xi, \qquad \gamma>2\xi.
\end{equation*}
For $S_1$, by Lemma \ref{expo-pair-gernal} with $(\kappa,\ell)=(\frac{1}{2},\frac{1}{2})$, we have
\begin{align*}
          S_1
 = & \,\, \sum_{d\leqslant D}\sum_{0<|h|\leqslant H_1}\frac{1}{2\pi ih}\sum_{n\sim X}
          \bigg(e\bigg(\frac{hn^{1/\gamma}}{d}\bigg)-e\bigg(\frac{h(n^{1/\gamma}-1)}{d}\bigg)\bigg)
                       \nonumber \\
 = & \,\, \sum_{d\leqslant D}\sum_{0<|h|\leqslant H_1}\frac{1-e(-h/d)}{2\pi ih}\sum_{n\sim X}
          e\bigg(\frac{hn^{1/\gamma}}{d}\bigg)
                       \nonumber \\
\ll & \,\, \sum_{d\leqslant D}\sum_{0<|h|\leqslant H_1}\frac{1}{d}\bigg|\sum_{n\sim X}
           e\bigg(\frac{hn^{1/\gamma}}{d}\bigg)\bigg|
                       \nonumber \\
\ll & \,\, \sum_{d\leqslant D}\sum_{0<|h|\leqslant H_1}\frac{1}{d}\bigg(\frac{d}{|h|}X^{1-1/\gamma}+
           \bigg(\frac{|h|}{d}\bigg)^{1/2}X^{1/(2\gamma)}\bigg)
                       \nonumber \\
\ll & \,\, \mathscr{L}DX^{1-1/\gamma}+H_1^{3/2}X^{1/(2\gamma)}
                       \nonumber \\
\ll & \,\, x^{\xi+\eta}+x^{3\xi/2+\eta}\cdot x^{(1-\eta)/2}\ll x^\gamma\mathscr{L}^{-A},
\end{align*}
provided that
\begin{equation}\label{suffi-condi-1}
 \gamma>\frac{1}{2}+\frac{3}{2}\xi.
\end{equation}
Therefore, it suffices to show that (\ref{mean-value-split}) still holds for
$x^{\gamma(1-\eta)}<X\leqslant x^{\gamma}$. It is easy to see that, for $\xi\leqslant\gamma(1-\eta)/2$, there holds
\begin{equation*}
 X^{\xi/\gamma}\leqslant D\leqslant X^{\xi/\gamma+\eta/2}.
\end{equation*}
Putting (\ref{psi-expan}) into the left--hand side of (\ref{mean-value-split}) with $H=H_2:=X^{\xi/\gamma+\eta}$, the contribution of the error term in (\ref{psi-expan}) to the left--hand side of (\ref{mean-value-split}) is
\begin{equation*}
 \sum_{d\leqslant D}\sum_{n\sim X}\Lambda(n)\bigg(g\bigg(\frac{n^{1/\gamma}}{d},H_2\bigg)
 +g\bigg(\frac{n^{1/\gamma}-1}{d},H_2\bigg)\bigg)=E_1+E_2,
\end{equation*}
say. For $E_1$ and $E_2$, uniformly, by Lemma \ref{expo-pair-gernal} with $(\kappa,\ell)=(\frac{1}{2},\frac{1}{2})$, we have
\begin{align*}
             E_1,E_2
 \ll & \,\, \mathscr{L}\sum_{d\leqslant D}\sum_{h=-\infty}^\infty|b_h|
            \bigg|\sum_{n\sim X}e\bigg(\frac{hn^{1/\gamma}}{d}\bigg)\bigg|
                       \nonumber \\
 \ll & \,\, \mathscr{L}\sum_{d\leqslant D}\Bigg(|b_0|X+\sum_{\substack{h=-\infty\\ h\not=0}}^\infty
            |b_h|\bigg(\frac{d}{|h|}X^{1-1/\gamma}+\bigg(\frac{|h|}{d}\bigg)^{1/2}X^{1/(2\gamma)}\bigg)\Bigg)
                       \nonumber \\
 \ll & \,\, \mathscr{L}^2DXH_2^{-1}+\mathscr{L}X^{1-1/\gamma}\sum_{d\leqslant D}d
            \sum_{\substack{h=-\infty\\ h\not=0}}^\infty\frac{1}{h^2}
                       \nonumber \\
     & \,\, +\mathscr{L}X^{1/(2\gamma)}\sum_{d\leqslant D}d^{-1/2}\Bigg(\sum_{0<|h|\leqslant H_2}|h|^{-1/2}+
            \sum_{|h|>H_2}\frac{H_2}{|h|^{3/2}}\Bigg)
                       \nonumber \\
 \ll & \,\, \mathscr{L}^2DXH_2^{-1}+\mathscr{L}D^2X^{1-1/\gamma}+\mathscr{L}X^{1/(2\gamma)}D^{1/2}H_2^{1/2}
            \ll x^{\gamma}\mathscr{L}^{-A},
\end{align*}
provided that
\begin{equation}\label{suffi-condi-2}
\gamma>\frac{1}{2}+\xi,\qquad \gamma>2\xi.
\end{equation}
Thus, it remains to show that
\begin{equation}\label{L-X-Main}
\mathcal{S}:=\sum_{d\leqslant D}\bigg|\sum_{0<|h|\leqslant H_2}\frac{1}{h}\sum_{n\sim X}\Lambda(n)
\bigg(e\bigg(\frac{hn^{1/\gamma}}{d}\bigg)-e\bigg(\frac{h(n^{1/\gamma}-1)}{d}\bigg)\bigg)\bigg|\ll x^{\gamma}\mathscr{L}^{-A}.
\end{equation}
For the left--hand side of (\ref{L-X-Main}), one has
\begin{align*}
           \mathcal{S}
 = & \,\,  \sum_{d\leqslant D}\bigg|\sum_{0<|h|\leqslant H_2}\frac{1-e(-h/d)}{h}\sum_{n\sim X}\Lambda(n)
           e\bigg(\frac{hn^{1/\gamma}}{d}\bigg)\bigg|
                  \nonumber \\
\ll & \,\, \sum_{d\leqslant D}\frac{1}{d}\sum_{0<h\leqslant H_2}\bigg|\sum_{n\sim X}\Lambda(n)
           e\bigg(\frac{hn^{1/\gamma}}{d}\bigg)\bigg|
                  \nonumber \\
  = & \,\, \sum_{d\leqslant D}\frac{1}{d}\sum_{0<h\leqslant H_2}\delta(d,h)\sum_{n\sim X}\Lambda(n)
           e\bigg(\frac{hn^{1/\gamma}}{d}\bigg)
                  \nonumber \\
  = & \,\, \sum_{n\sim X}\Lambda(n)\sum_{0<h\leqslant H_2}\sum_{d\leqslant D}\frac{\delta(d,h)}{d}
           e\bigg(\frac{hn^{1/\gamma}}{d}\bigg)
                  \nonumber \\
  =: & \,\, \sum_{n\sim X}\Lambda(n)G(n),
\end{align*}
where
\begin{equation*}
 G(n)=\sum_{0<h\leqslant H_2}\sum_{d\leqslant D}\frac{\delta(d,h)}{d}e\bigg(\frac{hn^{1/\gamma}}{d}\bigg),\qquad
 |\delta(d,h)|=1,
\end{equation*}
and we use the estimate $|1-e(-h/d)|\ll |h|d^{-1}$. Accordingly, in order to establish the estimate (\ref{L-X-Main}), it is sufficient to show that
\begin{equation*}
 \Bigg|\sum_{n\sim X}\Lambda(n)G(n)\Bigg|\ll X\mathscr{L}^{-A}.
\end{equation*}
A special case of the identity of Heath--Brown \cite{Heath-Brown-1982} is given by
\begin{equation*}
 -\frac{\zeta'}{\zeta}=-\frac{\zeta'}{\zeta}(1-Z\zeta)^3-\sum_{j=1}^3\binom{3}{j}(-1)^jZ^j\zeta^{j-1}(-\zeta'),
\end{equation*}
where $Z=Z(s)=\sum\limits_{m\leqslant X^{1/3}}\mu(m)m^{-s}$. From this we can decompose $\Lambda(n)$ for $n\sim X$ as
\begin{equation*}
 \Lambda(n)=\sum_{j=1}^3\binom{3}{j}(-1)^{j-1}\sum_{m_1\cdots m_{2j}=n}\mu(m_1)\cdots\mu(m_j)\log m_{2j}.
\end{equation*}
Thus, for any arithmetic function $G(n)$, we can express $\sum\limits_{n\sim X}\Lambda(n)G(n)$ in terms of sums
\begin{equation*}
 \mathop{\sum\,\,\cdots\,\,\sum}_{\substack{m_1\cdots m_{2j}\sim X\\ m_i\sim M_i}}\mu(m_1)\cdots\mu(m_j)
 (\log m_{2j})G(m_1\cdots m_{2j}),
\end{equation*}
where $1\leqslant j\leqslant3,\, M_1M_2\cdots M_{2j}\sim X$ and $M_1,\dots,M_j\leqslant X^{1/3}$. By dividing the $M_j$ into
two groups, we have
\begin{equation}\label{expo-fenjie}
 \Bigg|\sum_{n\sim X}\Lambda(n)G(n)\Bigg|\ll_\eta X^\eta\max\Bigg|\mathop{\sum\sum}_{\substack{m\ell\sim X\\ m\sim M}}
 a(m)b(\ell)G(mn)\Bigg|,
\end{equation}
where the maximum is taken over all bilinear forms with coefficients satisfying one of
\begin{equation}\label{type-II-coeff}
   |a(m)|\leqslant 1,\qquad \qquad |b(\ell)|\leqslant1,
\end{equation}
or
\begin{equation*}
   |a(m)|\leqslant 1,\qquad \qquad b(\ell)=1,
\end{equation*}
or
\begin{equation*}
   |a(m)|\leqslant 1,\qquad \qquad b(\ell)=\log \ell,
\end{equation*}
and also satisfying in all cases
\begin{equation}\label{gene-coeff-condi}
  M\leqslant X.
\end{equation}
We refer to the case (\ref{type-II-coeff}) as being Type II sums and to the other cases as being Type I sums and write for
brevity $\Sigma_{II}$ and $\Sigma_{I}$, respectively. By dividing the $M_j$ into two groups in a judicious fashion we are able to reduce the range of $M$ from (\ref{gene-coeff-condi}). In Section \ref{ex-section}, we shall give the estimate of these sums.

\section{Estimate of Exponential Sums}\label{ex-section}
In this section, we shall give the estimate of exponential sums which will be used in
proving Theorem \ref{Theorem-1}.

\subsection{The Estimate of Type II Sums}\label{subse-type-II}
We commence by breaking up the ranges for $\ell$ and $h$ into intervals $(L/2,L]$ and $(J/2,J]$ so
that $ML\asymp X$ and $1\ll J\leqslant H_2$. Then, for the Type II sums $\Sigma_{II}$, by splitting arguments, there holds
\begin{equation*}
\Sigma_{II}\ll\mathscr{L}^3\sum_{m\sim M}\Bigg|\sum_{\substack{\ell\sim L\\ m\ell\sim X}}b(\ell)
\sum_{h\sim J}\sum_{d\sim D}\frac{\delta(d,h)}{d}e\bigg(\frac{h(m\ell)^{1/\gamma}}{d}\bigg)\Bigg|.
\end{equation*}
Then we have
\begin{equation*}
  0<\frac{h\ell^{1/\gamma}}{d}\leqslant\frac{2JL^{1/\gamma}}{D}.
\end{equation*}
Denote by $T$ a parameter, which will be chosen later. We decompose the collection of available triples $(h,\ell,d)$ into sets
$\mathscr{S}_t\,(1\leqslant t\leqslant T)$, defined by
\begin{equation*}
\mathscr{S}_t=\Bigg\{(h,\ell,d):\,\,h\sim J,\, \ell\sim L, \,d\sim D,\,
\frac{2JL^{1/\gamma}(t-1)}{DT}<\frac{h\ell^{1/\gamma}}{d}\leqslant\frac{2JL^{1/\gamma}t}{DT}\Bigg\}.
\end{equation*}
Therefore, we have
\begin{equation*}
  \Sigma_{II}\ll \mathscr{L}^3\sum_{1\leqslant t\leqslant T}\sum_{m\sim M}
  \Bigg|\mathop{\sum\sum\sum}_{\substack{(h,\ell,d)\in\mathscr{S}_t\\
                 m\ell\sim X}}b(\ell)\frac{\delta(d,h)}{d}e\bigg(\frac{h(m\ell)^{1/\gamma}}{d}\bigg)\Bigg|,
\end{equation*}
which combined with Cauchy's inequality yields
\begin{align*}
            |\Sigma_{II}|^2
 \ll & \,\, \mathscr{L}^6TM\sum_{1\leqslant t\leqslant T}\sum_{m\sim M}
            \Bigg|\mathop{\sum\sum\sum}_{\substack{(h,\ell,d)\in\mathscr{S}_t\\ m\ell\sim X}}
            b(\ell)\frac{\delta(d,h)}{d}e\bigg(\frac{h(m\ell)^{1/\gamma}}{d}\bigg)\Bigg|^2
                   \nonumber \\
 \ll & \,\, \mathscr{L}^6TM\sum_{1\leqslant t\leqslant T}\mathop{\sum\sum\sum}_{(h_1,\ell_1,d_1)\in\mathscr{S}_t}
            \mathop{\sum\sum\sum}_{(h_2,\ell_2,d_2)\in\mathscr{S}_t}\frac{1}{d_1d_2}\Bigg|
            \sum_{\substack{m\sim M\\ m\ell_1\sim X \\m\ell_2\sim X}}
            e\bigg(\bigg(\frac{h_1\ell_1^{1/\gamma}}{d_1}-\frac{h_2\ell_2^{\gamma}}{d_2}\bigg)
            m^{1/\gamma}\bigg)\Bigg|
                   \nonumber \\
 \ll & \,\, \mathscr{L}^6TMD^{-2}\mathop{\sum_{h_1\sim J}\sum_{h_2\sim J}\sum_{\ell_1\sim L}\sum_{\ell_2\sim L}
            \sum_{d_1\sim D}\sum_{d_2\sim D}}_{|\lambda|\leqslant2JL^{1/\gamma}D^{-1}T^{-1}}\Bigg|
            \sum_{\substack{m\sim M\\ m\ell_1\sim X \\m\ell_2\sim X}}e(\lambda m^{1/\gamma})\Bigg|
\end{align*}
where
\begin{equation*}
  \lambda=\frac{h_1\ell_1^{1/\gamma}}{d_1}-\frac{h_2\ell_2^{1/\gamma}}{d_2}.
\end{equation*}
By Lemma \ref{expo-pair-gernal} with $(\kappa,\ell)=BA^3B(0,1)=(\frac{11}{30},\frac{8}{15})$, we obtain
\begin{align*}
 |\Sigma_{II}|^2\ll \mathscr{L}^6TMD^{-2}\mathop{\sum_{h_1\sim J}\sum_{h_2\sim J}\sum_{\ell_1\sim L}
 \sum_{\ell_2\sim L}\sum_{d_1\sim D}\sum_{d_2\sim D}}_{|\lambda|\leqslant2JL^{1/\gamma}D^{-1}T^{-1}}
 \min\Big\{M,|\lambda|^{-1}M^{1-1/\gamma}+|\lambda|^{11/30}M^{11/(30\gamma)+1/6}\Big\}
\end{align*}
If $|\lambda|\leqslant M^{-1/\gamma}$, then $M\ll|\lambda|^{-1}M^{1-1/\gamma}$. By Lemma \ref{latticepoints}, the contribution of the $M$ term to $|\Sigma_{II}|^2$ is
\begin{align}\label{lambda-s-upp}
 \ll & \,\, \mathscr{L}^6TM^2D^{-2}\mathop{\sum_{h_1\sim J}\sum_{h_2\sim J}\sum_{\ell_1\sim L}
            \sum_{\ell_2\sim L}\sum_{d_1\sim D}\sum_{d_2\sim D}}_{|\lambda|\leqslant M^{-1/\gamma}}1
                        \nonumber \\
 \ll & \,\, X^\eta TM^2D^{-2}\big(JDL+M^{-1/\gamma}JD^3L^{2-1/\gamma}\big)
                        \nonumber \\
 \ll & \,\, X^\eta \big(XMTJD^{-1}+X^{2-1/\gamma}TJD\big).
\end{align}
If $|\lambda|>M^{-1/\gamma}$, then $M>|\lambda|^{-1}M^{1-1/\gamma}$. It follows from the splitting
argument that the contribution of the $|\lambda|^{-1}M^{1-1/\gamma}$ term to $|\Sigma_{II}|^2$ is
\begin{align}\label{lambda-l-upp}
 \ll & \,\, \mathscr{L}^7TM^{2-1/\gamma}D^{-2}
            \max_{M^{-1/\gamma}\leqslant\Delta\leqslant2JL^{1/\gamma}D^{-1}T^{-1}}\mathscr{N}(\Delta)\Delta^{-1}
                    \nonumber \\
 \ll & \,\, X^\eta TM^{2-1/\gamma}D^{-2}\big(JDLM^{1/\gamma}+JD^3L^{2-1/\gamma}\big)
                    \nonumber \\
 \ll & \,\, X^\eta\big(XMTJD^{-1}+X^{2-1/\gamma}TJD\big).
\end{align}
Moreover, the total contribution of the term $|\lambda|^{11/30}M^{11/(30\gamma)+1/6}$ to $|\Sigma_{II}|^2$ is
\begin{align*}
\ll & \,\, |\lambda|^{11/30}M^{11/(30\gamma)+1/6}\cdot\mathscr{L}^6TMD^{-2}
           \cdot\mathscr{N}(2JL^{1/\gamma}D^{-1}T^{-1})
                \nonumber \\
\ll & \,\, X^\eta\big(JL^{1/\gamma}D^{-1}T^{-1}\big)^{11/30}TM^{11/(30\gamma)+7/6}D^{-2}
           \big(JDL+JL^{1/\gamma}D^{-1}T^{-1}\cdot JD^3L^{2-1/\gamma}\big)
                \nonumber \\
\ll & \,\, X^\eta\big(X^{11/(30\gamma)+2}M^{-5/6}J^{71/30}D^{-11/30}T^{-11/30}+
           X^{11/(30\gamma)+1}M^{1/6}J^{41/30}D^{-41/30}T^{19/30}\big),
\end{align*}
which combined with (\ref{lambda-s-upp}) and (\ref{lambda-l-upp}) yields
\begin{align}\label{Sigma-2-fi-1}
           |\Sigma_{II}|^{2}
\ll & \,\, X^\eta\big(X^{2-1/\gamma}TJD+XMTJD^{-1}+X^{11/(30\gamma)+2}M^{-5/6}J^{71/30}D^{-11/30}T^{-11/30}
                 \nonumber \\
    & \,\, \qquad+X^{11/(30\gamma)+1}M^{1/6}J^{41/30}D^{-41/30}T^{19/30}\big).
\end{align}
Now, we choose
\begin{equation}\label{T-chosen}
 T=\big[X^{(30\gamma+11)/(41\gamma)}M^{-55/41}D^{60/41}\big].
\end{equation}
Trivially, one has $T>1$ provided that $M\ll X^{(11+60\xi+30\gamma)/(55\gamma)-\eta}$. Inserting (\ref{T-chosen}) into (\ref{Sigma-2-fi-1}),  we obtain
\begin{align*}
       |\Sigma_{II}|^2
\ll & \,\, X^\eta\Big(M^{-55/41}X^{(112\gamma-30)/(41\gamma)}JD^{101/41}+
           M^{-14/41}X^{(71\gamma+11)/(41\gamma)}JD^{19/41}
                   \nonumber \\
    & \,\,  +M^{-14/41}X^{(71\gamma+11)/(41\gamma)}J^{71/30}D^{-1111/1230}
           +M^{-28/41}X^{(60\gamma+22)/(41\gamma)}J^{41/30}D^{-541/1230}\Big),
\end{align*}
which combined with $J\ll H_2=X^{\xi/\gamma+\eta}$ and $X^{\xi/\gamma}\leqslant D\leqslant X^{\xi/\gamma+\eta/2}$ yields
\begin{align*}
           |\Sigma_{II}|^2
\ll & \,\, X^\eta\Big(M^{-55/41}X^{(112\gamma+142\xi-30)/(41\gamma)}+M^{-14/41}X^{(71\gamma+60\xi+11)/(41\gamma)}
                 \nonumber \\
    & \,\, \qquad +M^{-28/41}X^{(60\gamma+38\xi+22)/(41\gamma)}\Big).
\end{align*}
According to above arguments, we deduce the following lemma.

\begin{lemma}\label{Type-II-es}
Suppose that $\frac{1}{2}<\gamma<1$ and $0<\xi\leqslant\gamma(1-\eta)/2$ satisfy the condition
\begin{equation}\label{Type-II-condi}
 \gamma>\frac{11}{25}+\frac{12}{5}\xi+\eta.
\end{equation}
If there holds
\begin{equation*}
 X^{(11+60\xi-11\gamma)/(14\gamma)+\eta}\ll M\ll X^{(11+60\xi+30\gamma)/(55\gamma)-\eta},
\end{equation*}
then we have
\begin{equation*}
 \Sigma_{II}\ll X^{1-\eta}.
\end{equation*}
\end{lemma}

\subsection{The Estimate of Type I Sums}
As in Subsection \ref{subse-type-II}, we begin by breaking up the range for $\ell$ into intervals $(L/2,L]$,
such that $MN\asymp X$. Then we change the order of summation and derive that
\begin{equation*}
\Sigma_I\ll\sum_{d\leqslant D}\frac{1}{d}\sum_{0<|h|\leqslant H_2}\sum_{m\sim M}\sum_{\ell\sim L}a(m)b(\ell)
\delta(d,h)e\bigg(\frac{h(m\ell)^{1/\gamma}}{d}\bigg).
\end{equation*}
According to the splitting argument, it suffices to show that, for fixed $d\sim D$, there holds
\begin{equation}\label{Type-I-inner}
\sum_{h\sim H_2}\sum_{m\sim M}\sum_{\ell\sim L}a(m)b(\ell)\delta(d,h)e\bigg(\frac{h(m\ell)^{1/\gamma}}{d}\bigg)
\ll X^{1-\eta}.
\end{equation}
Denote by $\mathcal{K}_d$ the left--hand side of (\ref{Type-I-inner}). By partial summation, one has
\begin{align*}
\mathcal{K}_d=\sum_{h\sim H_2}\sum_{m\sim M}a(m)\delta(d,h)\sum_{\ell\sim L}b(\ell)
          e\bigg(\frac{h(m\ell)^{1/\gamma}}{d}\bigg)
          \ll \mathscr{L}\sum_{h\sim H_2}\sum_{m\sim M}\Bigg|\sum_{\ell\sim L}
           e\bigg(\frac{h(m\ell)^{1/\gamma}}{d}\bigg)\Bigg|.
\end{align*}
It follows from Lemma \ref{Robert-Sargos-lemma} with parameters $(N,H,M)=(H_2,M,L)$ that
\begin{equation*}
\mathcal{K}_d\ll(MH_2L)^{1+\eta}\bigg(\bigg(\frac{Y}{MH_2L^2}\bigg)^{1/4}+\frac{1}{L^{1/2}}+\frac{1}{Y}\bigg),
\end{equation*}
where $Y=d^{-1}H_2M^{1/\gamma}L^{1/\gamma}$. Since $d\sim D\ll H_2$, we have $Y^{-1}\ll L^{-1/2}$. If
$M\ll X^{(1-\xi)/\gamma-\eta}$, then
\begin{equation*}
 \frac{1}{L^{1/2}}\ll\bigg(\frac{Y}{MH_2L^2}\bigg)^{1/4}.
\end{equation*}
Accordingly, under the condition $M\ll X^{(1-\xi)/\gamma-\eta}$, there holds
\begin{align*}
            \mathcal{K}_d
 \ll & \,\, X^\eta(MH_2L)\bigg(\frac{Y}{MH_2L^2}\bigg)^{1/4}
            \ll M^{(3\gamma+1)/(4\gamma)}L^{(2\gamma+1)/(4\gamma)}H_2d^{-1/4}
                     \nonumber \\
 \ll & \,\, X^{1/(4\gamma)+1/2}M^{1/4}H_2D^{-1/4}\ll X^{(3\xi+1)/(4\gamma)+1/2+\eta}M^{1/4}\ll X^{1-\eta},
\end{align*}
provided that $M\ll X^{2-(3\xi+1)/\gamma-\eta}$. By noting that $(1-\xi)/\gamma>2-(3\xi+1)/\gamma$, we obtain the
following lemma.

\begin{lemma}\label{Type-I-es}
Suppose that M satisfies the condition
\begin{equation*}
 M\ll X^{2-(3\xi+1)/\gamma-\eta}.
\end{equation*}
Then we have
\begin{equation*}
 \Sigma_{I}\ll X^{1-\eta}.
\end{equation*}
\end{lemma}
Now, we shall finish establishing the mean value theorem (\ref{mean-value-equilavent}). To deal with the right--hand side of (\ref{expo-fenjie}), the following lemma is required.
\begin{lemma}\label{exponen-fenjie}
If we have real numbers $0<\mathfrak{a}<1,\,0<\mathfrak{b}<\mathfrak{c}<1$  satisfying
\begin{equation*}
\mathfrak{b}<\frac{2}{3},\qquad 1-\mathfrak{c}<\mathfrak{c}-\mathfrak{b}, \qquad 1-\mathfrak{a}<\frac{\mathfrak{c}}{2},
\end{equation*}
then (\ref{expo-fenjie}) still holds when (\ref{gene-coeff-condi}) is replaced by the conditions
\begin{equation*}
M\leqslant X^\mathfrak{a} \qquad \textrm{for Type I sums},
\end{equation*}
and
\begin{equation*}
X^\mathfrak{b}\leqslant M\leqslant X^\mathfrak{c} \qquad \textrm{for Type II sums}.
\end{equation*}
\end{lemma}
\begin{proof}
 See Proposition 1 of Balog and Friedlander \cite{Balog-Friedlander-1992}.  $\hfill$
\end{proof}
Combining Lemma \ref{Type-II-es}, Lemma \ref{Type-I-es} and Lemma \ref{exponen-fenjie}, we take
\begin{align*}
\mathfrak{a}= & \,\, 2-(3\xi+1)/\gamma-\eta,
                         \nonumber \\
\mathfrak{b}= & \,\, (11+60\xi-11\gamma)/(14\gamma)+\eta,
                         \nonumber \\
\mathfrak{c}= & \,\, (11+60\xi+30\gamma)/(55\gamma)-\eta.
\end{align*}
Thus, we are able to choose
\begin{equation}\label{level-def}
 \xi=\xi(\gamma)=(140\gamma-99)/270-\eta.
\end{equation}
For $99/140<\gamma<1$, it is easy to check the conditions  (\ref{suffi-condi-1}), (\ref{suffi-condi-2}) and (\ref{Type-II-condi}), as well as the inequalities in Lemma \ref{exponen-fenjie}, hold. Accordingly, (\ref{mean-value-equilavent}) holds for $\xi$ defined as in (\ref{level-def}).

\section{Proof of Theorem \ref{Theorem-1}}
 In this section, we shall prove Theorem \ref{Theorem-1} according to the weighted sieve of
Richert \cite{Richert-1969} and the method of Chen \cite{Chen-1973}.

We consider the weighted sum
\begin{equation*}
 W\big(\mathscr{A},x^{1/17.41}\big)=\sum_{\substack{a\in\mathscr{A}\\ (a,P(x^{1/17.41}))=1}}
 \Bigg(1-\lambda\sum_{\substack{x^{1/17.41}\leqslant p<x^{1/u}\\ p|a}}\bigg(1-\frac{u\log p}{\log x}\bigg)\Bigg),
\end{equation*}
where $\lambda=(9-u-\varepsilon)^{-1},u=\xi^{-1}+\varepsilon$, $\xi$ is defined as in (\ref{level-def}), and
\begin{equation*}
 P(z)=\prod_{p<z}p.
\end{equation*}
For convenience, we write
\begin{equation*}
 \mathscr{W}_a=1-\lambda\sum_{\substack{x^{1/17.41}\leqslant p<x^{1/u}\\ p|a}}\bigg(1-\frac{u\log p}{\log x}\bigg).
\end{equation*}
Then we have
\begin{equation}\label{W-fenjie}
     W\big(\mathscr{A},x^{1/17.41}\big)
   =\sum_{\substack{a\in\mathscr{A}\\ (a,P(x^{1/17.41}))=1\\ \Omega(a)\leqslant7}}\mathscr{W}_a
   +\sum_{\substack{a\in\mathscr{A}\\ (a,P(x^{1/17.41}))=1\\ \Omega(a)=8\\ \mu(a)\not=0}}\mathscr{W}_a
   +\sum_{\substack{a\in\mathscr{A}\\ (a,P(x^{1/17.41}))=1\\ \Omega(a)\geqslant9\\ \mu(a)\not=0}}\mathscr{W}_a
   +\sum_{\substack{a\in\mathscr{A}\\ (a,P(x^{1/17.41}))=1\\ \Omega(a)\geqslant8\\ \mu(a)=0}}\mathscr{W}_a.
\end{equation}
Trivially, we have
\begin{align}\label{n-s-free>8}
             \sum_{\substack{a\in\mathscr{A}\\ (a,P(x^{1/17.41}))=1\\ \Omega(a)\geqslant8\\ \mu(a)=0}}\mathscr{W}_a
 \ll  & \,\, \sum_{\substack{a\in\mathscr{A}\\ (a,P(x^{1/17.41}))=1\\ \mu(a)=0}}\tau(a)
             \ll x^\varepsilon \sum_{x^{1/17.41}\leqslant p_1\leqslant x^{1/2}}
             \sum_{\substack{p\leqslant x^\gamma\\ [p^{1/\gamma}]\equiv0\!\!\!\!\!\pmod {p_1^2}}}1
                                 \nonumber \\
 \ll  & \,\, x^\varepsilon\sum_{x^{1/17.41}\leqslant p_1\leqslant x^{1/2}}\bigg(\frac{x}{p_1^2}+1\bigg)
         \ll  x^\varepsilon\big(x^{1-1/17.41}+x^{1/2}\big)\ll x^{16.41/17.41+\varepsilon}.
\end{align}
For given integer $a$ with $a\leqslant x,(a,P(x^{1/17.41}))=1$ and $\mu(a)\not=0$, the weight $\mathscr{W}_a$ in the sum $W(\mathscr{A},x^{1/17.41})$ satisfies
\begin{align}\label{weight-upper}
           1-\lambda\sum_{\substack{x^{1/17.41}\leqslant p<x^{1/u}\\ p|a}}\bigg(1-\frac{u\log p}{\log x}\bigg)
\leqslant & \,\,\lambda\bigg(\frac{1}{\lambda}-\sum_{p|a}\bigg(1-\frac{u\log p}{\log x}\bigg)\bigg)
                     \nonumber \\
= & \,\, \lambda\bigg(9-u-\varepsilon-\Omega(a)+\frac{u\log a}{\log x}\bigg)<\lambda(9-\Omega(a)),
\end{align}
and thus $\mathscr{W}_a<0$ for $\Omega(a)\geqslant9$. From (\ref{W-fenjie})--(\ref{weight-upper}), we know that
\begin{align}\label{omega(a)<7-lower}
        \sum_{\substack{a\in\mathscr{A}\\ (a,P(x^{1/17.41}))=1\\ \Omega(a)\leqslant7}}\mathscr{W}_a
= &\,\, W\big(\mathscr{A},x^{1/17.41}\big)-\sum_{\substack{a\in\mathscr{A}\\ (a,P(x^{1/17.41}))=1\\ \Omega(a)=8\\
        \mu(a)\not=0}}\mathscr{W}_a-\sum_{\substack{a\in\mathscr{A}\\ (a,P(x^{1/17.41}))=1\\ \Omega(a)\geqslant9\\ \mu(a)\not=0}}\mathscr{W}_a+O\big(x^{16.41/17.41+\varepsilon}\big)
                     \nonumber \\
\geqslant &\,\, W\big(\mathscr{A},x^{1/17.41}\big)-\sum_{\substack{a\in\mathscr{A}\\ (a,P(x^{1/17.41}))=1\\
                \Omega(a)=8\\ \mu(a)\not=0}}\mathscr{W}_a+O\big(x^{16.41/17.41+\varepsilon}\big).
\end{align}
Therefore, if we can show that the contribution of the second term on the right--hand side of (\ref{omega(a)<7-lower}) is strictly less than $(1-\varepsilon)W(\mathscr{A},x^{1/17.41})$, then we shall prove Theorem \ref{Theorem-1}. For $W(\mathscr{A},x^{1/17.41})$, we have
\begin{align}\label{W(a)-chai}
         W\big(\mathscr{A},x^{1/17.41}\big)
= & \,\, \sum_{\substack{a\in\mathscr{A}\\(a,P(x^{1/17.41}))=1}}1-\lambda\sum_{x^{1/17.41}\leqslant p<x^{1/u}}
         \bigg(1-\frac{u\log p}{\log x}\bigg)\sum_{\substack{a\in\mathscr{A}\\(a,P(x^{1/17.41}))=1\\ p|a}}1
               \nonumber\\
= & \,\, S\big(\mathscr{A},x^{1/17.41}\big)-\lambda\sum_{x^{1/17.41}\leqslant p<x^{1/u}}
         \bigg(1-\frac{u\log p}{\log x}\bigg)S\big(\mathscr{A}_p,x^{1/17.41}\big).
\end{align}
Now, we shall use Lemma \ref{upper-lower-sieve} to give the lower bound of
$S\big(\mathscr{A},x^{1/17.41}\big)$. According to (\ref{A_d-asymp}), we take
\begin{equation*}
 X=\pi(x^\gamma),\qquad \omega(d)=
 \begin{cases}
   1, & \textrm{if $\mu(d)\not=0$ and $d|P(x^{1/17.41})$},\\
   0, & \textrm{otherwise}.
 \end{cases}
\end{equation*}
By (\ref{mean-value-equilavent}) and the arguments in Section \ref{ex-section}, we know that
\begin{equation*}
\sum_{d\leqslant x^\xi}\bigg|\#\mathscr{A}_d-\frac{\omega(d)}{d}\pi(x^\gamma)\bigg|\ll\frac{x^\gamma}{(\log x)^A},
\end{equation*}
where $\xi$ is defined by (\ref{level-def}). It follows from (\ref{diff-eq}) that
\begin{equation*}
  F(s)=\frac{2e^{C_0}}{s},\qquad 0<s\leqslant3;\qquad f(s)=\frac{2e^{C_0}\log(s-1)}{s}, \qquad 2\leqslant s\leqslant4,
\end{equation*}
where $C_0$ denotes Euler's constant. By noting the fact that $2<17.41\xi<4$ holds for $0.93<\gamma<1$, then
Lemma \ref{upper-lower-sieve} gives
\begin{align}
                  S\big(\mathscr{A},x^{1/17.41}\big)
\geqslant & \,\, \pi(x^\gamma)V\big(x^{1/17.41}\big)\Big(f\big(17.41\xi\big)-o(1)\Big)
                     \nonumber \\
= & \,\, 2e^{C_0}\pi(x^\gamma)V\big(x^{1/17.41}\big)\Bigg(\frac{\log\big(17.41\xi-1\big)}{17.41\xi}-o(1)\Bigg),
\end{align}
where $C_0$ denotes Euler's constant.
Moreover, it follows from (1.11) on p. 245  and (1.13) on p. 246 of Halberstam and Richert
\cite{Halberstam-Richert-book} that
\begin{align}\label{A_p-W-lower}
   & \,\, \sum_{x^{1/17.41}\leqslant p<x^{1/u}}\bigg(1-\frac{u\log p}{\log x}\bigg)
          S\big(\mathscr{A}_p,x^{1/17.41}\big)
                  \nonumber \\
\leqslant & \,\, \pi(x^\gamma)V\big(x^{1/17.41}\big)\Bigg(\sum_{x^{1/17.41}\leqslant p<x^{1/u}}
                 \bigg(1-\frac{u\log p}{\log x}\bigg)\frac{1}{p}\cdot
                 F\bigg(\frac{\log(x^\xi/p)}{\log x^{1/17.41}}\bigg)+o(1)\Bigg)
                         \nonumber \\
= & \,\, 2e^{C_0}\pi(x^\gamma)V\big(x^{1/17.41}\big)\bigg(\int_u^{17.41}\frac{t-u}{17.41t(\xi t -1)}
           \mathrm{d}t+o(1)\bigg).
\end{align}
Combining (\ref{W(a)-chai})--(\ref{A_p-W-lower}), we obtain
\begin{equation}\label{W(a,1/17.41)-lower}
   W\big(\mathscr{A},x^{1/17.41}\big)\geqslant2e^{C_0}\pi(x^\gamma)V\big(x^{1/17.41}\big)
   \Bigg(\frac{\log\big(17.41\xi-1\big)}{17.41\xi}-\lambda\int_u^{17.41}\frac{t-u}{17.41t(\xi t-1)}
           \mathrm{d}t+o(1)\Bigg).
\end{equation}
Now, we consider the second term on the right--hand side of (\ref{omega(a)<7-lower}). Set
\begin{equation*}
   \mathscr{B}=\Big\{m:m\leqslant x,\,m=p_1p_2\dots p_8,\,\, x^{1/17.41}\leqslant p_1<p_2<\dots<p_8\Big\},
\end{equation*}
and
\begin{equation*}
   \mathscr{E}=\Big\{n:\,[n^{1/\gamma}]\in\mathscr{B}\Big\}.
\end{equation*}
From (\ref{weight-upper}) we deduce that
\begin{align}\label{E-trans-upper}
\sum_{\substack{a\in\mathscr{A}\\ (a,P(x^{1/17.41}))=1\\ \Omega(a)=8\\ \mu(a)\not=0}}\mathscr{W}_a
< & \,\, \lambda \sum_{\substack{a\in\mathscr{A}\\ (a,P(x^{1/17.41}))=1\\ \Omega(a)=8\\ \mu(a)\not=0}}1
=\lambda\sum_{\substack{[p^{1/\gamma}]\leqslant x\\ [p^{1/\gamma}]=p_1p_2\dots p_8\\ x^{1/17.41}\leqslant p_1<p_2<\dots<p_8}}1=\lambda\cdot S\big(\mathscr{E},x^{\gamma/2}\big).
\end{align}
Let $\mathscr{E}_d=\big\{n\in\mathscr{E}: n\equiv0\!\pmod d\big\}$. Then one has $\#\mathscr{E}_d=\#\big\{m:[(dm)^{1/\gamma}]\in\mathscr{B}\big\}$. By setting $[(dm)^{1/\gamma}]=\ell\in\mathscr{B}$, we get
\begin{equation*}
[(dm)^{1/\gamma}]=\ell \quad\Longleftrightarrow \quad\ell\leqslant(dm)^{1/\gamma}<\ell+1 \quad
\Longleftrightarrow  \quad -\frac{(\ell+1)^\gamma}{d}<-m\leqslant-\frac{\ell^\gamma}{d},
\end{equation*}
and thus
\begin{align*}
          \#\mathscr{E}_d
 = & \,\, \sum_{m:\,[(dm)^{1/\gamma}]\in\mathscr{B}}1=\sum_{\ell\in\mathscr{B}}
          \bigg(\bigg[-\frac{\ell^\gamma}{d}\bigg]-\bigg[-\frac{(\ell+1)^\gamma}{d}\bigg]\bigg)
                       \nonumber \\
 = & \,\, \frac{1}{d}\sum_{\ell\in\mathscr{B}}\big((\ell+1)^\gamma-\ell^\gamma\big)
          +\sum_{\ell\in\mathscr{B}}\bigg(\psi\bigg(-\frac{(\ell+1)^\gamma}{d}\bigg)-
          \psi\bigg(-\frac{\ell^\gamma}{d}\bigg)\bigg)=\frac{1}{d}\mathcal{X}+\mathscr{R}_d,
\end{align*}
where
\begin{equation}\label{8-main-error-def}
  \mathcal{X}=\sum_{\ell\in\mathscr{B}}\big((\ell+1)^\gamma-\ell^\gamma\big),\qquad
  \mathscr{R}_d=\sum_{\ell\in\mathscr{B}}\bigg(\psi\bigg(-\frac{(\ell+1)^\gamma}{d}\bigg)-
                \psi\bigg(-\frac{\ell^\gamma}{d}\bigg)\bigg).
\end{equation}
In order to use Lemma \ref{upper-lower-sieve} to give upper bound for $S(\mathscr{E},x^{\gamma/2})$, we need to establish the following lemma.
\begin{lemma}\label{omega=8-error}
Let $\mathscr{R}_d$ be defined as in (\ref{8-main-error-def}). Then, for any $A>0$, we have
\begin{equation}
  \sum_{d\leqslant x^\xi}\big|\mathscr{R}_d\big|\ll \frac{x^\gamma}{(\log x)^{A}}.
\end{equation}
\end{lemma}
\begin{proof}
From the definition of $\mathscr{R}_d$ and splitting argument, it is sufficient to show that, for $X\leqslant x$,
there holds
\begin{equation}\label{error-splitting-1}
\sum_{d\sim D}\Bigg|\sum_{\substack{\ell\in\mathscr{B}\\ \ell\sim X}}
\bigg(\psi\bigg(-\frac{(\ell+1)^\gamma}{d}\bigg)-\psi\bigg(-\frac{\ell^\gamma}{d}\bigg)\bigg)\Bigg|
\ll\frac{x^\gamma}{(\log x)^A},
\end{equation}
where $D=x^\xi$. Let $\mathds{1}_{\mathscr{B}}(\cdot)$ be the characteristic function supporting on $\mathscr{B}$. Then (\ref{error-splitting-1}) is equivalent to
\begin{equation}\label{error-splitting-2}
\sum_{d\sim D}\Bigg|\sum_{ \ell\sim X}\mathds{1}_{\mathscr{B}}(\ell)
\bigg(\psi\bigg(-\frac{(\ell+1)^\gamma}{d}\bigg)-\psi\bigg(-\frac{\ell^\gamma}{d}\bigg)\bigg)\Bigg|
\ll\frac{x^\gamma}{(\log x)^A}.
\end{equation}
If $X\leqslant x^{(1-\eta)}$, then the left--hand side of (\ref{error-splitting-2}) is
\begin{align}\label{8-error-f-1}
 = & \,\, \sum_{d\sim D}\Bigg|\sum_{\ell\sim X}\mathds{1}_{\mathscr{B}}(\ell)\bigg(
          \frac{\ell^\gamma-(\ell+1)^\gamma}{d}+\bigg[-\frac{\ell^\gamma}{d}\bigg]-
          \bigg[-\frac{(\ell+1)^\gamma}{d}\bigg]\bigg)\Bigg|
                 \nonumber \\
 \leqslant & \,\, \bigg(\sum_{d\sim D}\frac{1}{d}\bigg)\bigg(\sum_{\ell\sim X}
                  \big((\ell+1)^\gamma-\ell^\gamma\big)\bigg)+\sum_{d\sim D}\Bigg|\sum_{\ell\sim X}
                  \mathds{1}_{\mathscr{B}}(\ell)\bigg(\bigg[-\frac{\ell^\gamma}{d}\bigg]-
                  \bigg[-\frac{(\ell+1)^\gamma}{d}\bigg]\bigg)\Bigg|
                  \nonumber \\
 \leqslant & \,\, \bigg(\sum_{d\sim D}\frac{1}{d}\bigg)\bigg(\sum_{\ell\sim X}
                  \big((\ell+1)^\gamma-\ell^\gamma\big)\bigg)+\sum_{d\sim D}\sum_{\ell\sim X}
                  \bigg(\bigg[-\frac{\ell^\gamma}{d}\bigg]-\bigg[-\frac{(\ell+1)^\gamma}{d}\bigg]\bigg)
                  \nonumber \\
 = & \,\, 2\bigg(\sum_{d\sim D}\frac{1}{d}\bigg)\bigg(\sum_{\ell\sim X}
                  \big((\ell+1)^\gamma-\ell^\gamma\big)\bigg)+\sum_{d\sim D}\sum_{\ell\sim X}
                  \bigg(\psi\bigg(-\frac{(\ell+1)^\gamma}{d}\bigg)-\psi\bigg(-\frac{\ell^\gamma}{d}\bigg)\bigg)
                   \nonumber \\
 \ll & \,\, X^\gamma+\sum_{d\sim D}\sum_{\ell\sim X}
                  \bigg(\psi\bigg(-\frac{(\ell+1)^\gamma}{d}\bigg)-\psi\bigg(-\frac{\ell^\gamma}{d}\bigg)\bigg)
                  \nonumber \\
 \ll & \,\, x^{\gamma-\eta}+\sum_{d\sim D}\sum_{\ell\sim X}
                  \bigg(\psi\bigg(-\frac{(\ell+1)^\gamma}{d}\bigg)-\psi\bigg(-\frac{\ell^\gamma}{d}\bigg)\bigg).
\end{align}
For the sum on the right--hand side of (\ref{8-error-f-1}), by Lemma \ref{psi-expansion} with $H=H_*:=x^{1-\gamma+\xi+\eta}$ in (\ref{psi-expan}), we have
\begin{align*}
  & \,\, \sum_{d\sim D}\sum_{\ell\sim X}
         \bigg(\psi\bigg(-\frac{(\ell+1)^\gamma}{d}\bigg)-\psi\bigg(-\frac{\ell^\gamma}{d}\bigg)\bigg)
               \nonumber \\
= & \,\, \sum_{d\sim D}\sum_{\ell\sim X}\Bigg(\sum_{0<|h|\leqslant H_*}\frac{1}{2\pi ih}
         \bigg(e\bigg(-\frac{h\ell^\gamma}{d}\bigg)-e\bigg(-\frac{h(\ell+1)^\gamma}{d}\bigg)\bigg)
               \nonumber \\
  & \,\, +O\bigg(g\bigg(\frac{\ell^\gamma}{d},H_*\bigg)\bigg)
         +O\bigg(g\bigg(\frac{(\ell+1)^\gamma}{d},H_*\bigg)\bigg)\Bigg)=\Theta_1+\Theta_2+\Theta_3,
\end{align*}
say. We only need to deal with $\Theta_2$, since $\Theta_3$ can be treated exactly the same. By Lemma \ref{expo-pair-gernal} with $(\kappa,\ell)=(\frac{1}{2},\frac{1}{2})$, we have
\begin{align}\label{theta-2-deal}
            \Theta_2
 \ll & \,\, \sum_{d\sim D}\sum_{h=-\infty}^\infty|b_h|
            \Bigg|\sum_{\ell\sim X}e\bigg(\frac{h\ell^{\gamma}}{d}\bigg)\Bigg|
                       \nonumber \\
 \ll & \,\, \mathscr{L}XDH_*^{-1}+\sum_{d\sim D}\sum_{\substack{h=-\infty\\ h\not=0}}^\infty
            |b_h|\Bigg|\sum_{\ell\sim X}e\bigg(\frac{h\ell^\gamma}{d}\bigg)\Bigg|
                       \nonumber \\
 \ll & \,\, \mathscr{L}XDH_*^{-1}+\sum_{d\sim D}\sum_{\substack{h=-\infty\\ h\not=0}}^\infty
            |b_h|\bigg(\frac{d}{|h|}X^{1-\gamma}+\bigg(\frac{|h|}{d}\bigg)^{1/2}X^{\gamma/2}\bigg)
                       \nonumber \\
 \ll & \,\, \mathscr{L}XDH_*^{-1}+\sum_{d\sim D}\Bigg(\sum_{0<|h|\leqslant H_*}\frac{1}{|h|}
            \bigg(\frac{d}{|h|}X^{1-\gamma}+\bigg(\frac{|h|}{d}\bigg)^{1/2}X^{\gamma/2}\bigg)
                       \nonumber \\
     & \,\, +\sum_{|h|>H_*}\frac{H_*}{h^2}
            \bigg(\frac{d}{|h|}X^{1-\gamma}+\bigg(\frac{|h|}{d}\bigg)^{1/2}X^{\gamma/2}\bigg)\Bigg)
                       \nonumber \\
 \ll & \,\, \mathscr{L}XDH_*^{-1}+D^2X^{1-\gamma}+D^{1/2}H_*^{1/2}X^{\gamma/2}\ll x^{\gamma}\mathscr{L}^{-A},
\end{align}
provided that
\begin{equation*}
\gamma>\frac{1}{2}+\xi.
\end{equation*}
For $\Theta_1$, by partial summation and Lemma \ref{expo-pair-gernal} with
$(\kappa,\ell)=(\frac{1}{2},\frac{1}{2})$, we get
\begin{align*}
            \Theta_1
\ll & \,\, X^{\gamma-1}\sum_{d\sim D}\frac{1}{d}\sum_{0<|h|\leqslant H_*}
           \Bigg|\sum_{\ell\sim X}e\bigg(\frac{h\ell^{\gamma}}{d}\bigg)\Bigg|
                 \nonumber \\
\ll & \,\, X^{\gamma-1}\sum_{d\sim D}\frac{1}{d}\sum_{0<|h|\leqslant H_*}
           \bigg(\frac{d}{|h|}X^{1-\gamma}+\bigg(\frac{|h|}{d}\bigg)^{1/2}X^{\gamma/2}\bigg)
                 \nonumber \\
\ll & \,\, D\mathscr{L}+X^{1/2+3\xi/2+\eta}\ll x^{\gamma-\eta},
\end{align*}
provided that
\begin{equation*}
\gamma>\frac{1}{2}+\frac{3}{2}\xi.
\end{equation*}
Now, we assume that $x^{1-\eta}<X\leqslant x$, by (\ref{psi-expan}) we know that the total contribution of the error term in (\ref{psi-expan}) to the left--hand side of (\ref{error-splitting-2}) is
\begin{equation*}
\ll \sum_{d\sim D}\sum_{\ell\sim X}\bigg(g\bigg(\frac{(\ell+1)^\gamma}{d},H\bigg)+
g\bigg(\frac{\ell^\gamma}{d},H\bigg)\bigg):=\mathfrak{S}_1^*+\mathfrak{S}_2^*.
\end{equation*}
One can show that $\mathfrak{S}_1^*,\mathfrak{S}_2^*\ll x^\gamma\mathscr{L}^{-A}$ by following the same processes as those in (\ref{theta-2-deal}) under the conditions
\begin{equation*}
 H=X^{1-\gamma+\xi+\eta},\qquad \gamma>\frac{1}{2}+\xi.
\end{equation*}
The contribution of the main term in (\ref{psi-expan}) to the left--hand side of (\ref{error-splitting-2}) is
\begin{align*}
= & \,\,\sum_{d\sim D}\Bigg|\sum_{\ell\sim X}\mathds{1}_{\mathscr{B}}(\ell)\sum_{0<|h|\leqslant H}\frac{1}{2\pi ih}
    \bigg(e\bigg(-\frac{h\ell^\gamma}{d}\bigg)-e\bigg(-\frac{h(\ell+1)^\gamma}{d}\bigg)\bigg)\Bigg|
               \nonumber \\
\ll & \,\,\sum_{d\sim D}\sum_{0<h\leqslant H}\frac{1}{h}\Bigg|\sum_{\ell\sim X}\mathds{1}_{\mathscr{B}}(\ell)
      \bigg(e\bigg(-\frac{h\ell^\gamma}{d}\bigg)-e\bigg(-\frac{h(\ell+1)^\gamma}{d}\bigg)\bigg)\Bigg|
      =:\mathfrak{S},
\end{align*}
say. Consequently, it suffices to show that $\mathfrak{S}\ll x^\gamma\mathscr{L}^{-A}$. Define
\begin{equation*}
\mathfrak{f}_{h,d}(\ell)=1-e\bigg(\frac{h(\ell^\gamma-(\ell+1)^\gamma)}{d}\bigg).
\end{equation*}
It follows from partial summation that
\begin{align*}
          \mathfrak{S}
 = & \,\, \sum_{d\sim D}\sum_{0<h\leqslant H}\frac{1}{h}\Bigg|\sum_{\ell\sim X}\mathds{1}_{\mathscr{B}}(\ell)
          e\bigg(-\frac{h\ell^\gamma}{d}\bigg)\mathfrak{f}_{h,d}(\ell) \Bigg|
                \nonumber \\
 = & \,\, \sum_{d\sim D}\sum_{0<h\leqslant H}\frac{1}{h}\Bigg|\int_{\frac{X}{2}}^X\mathfrak{f}_{h,d}(u)
          \mathrm{d}\bigg(\sum_{X/2<\ell\leqslant u}\mathds{1}_{\mathscr{B}}(\ell)
          e\bigg(-\frac{h\ell^\gamma}{d}\bigg)\bigg)\Bigg|
                \nonumber \\
 \ll & \,\, \sum_{d\sim D}\sum_{0<h\leqslant H}\frac{1}{h}\Bigg(\big|\mathfrak{f}_{h,d}(X)\big|\Bigg|
            \sum_{\ell\sim X}\mathds{1}_{\mathscr{B}}(\ell)e\bigg(-\frac{h\ell^\gamma}{d}\bigg)\Bigg|
                \nonumber \\
     & \,\, +\int_{\frac{X}{2}}^X\Bigg|\sum_{X/2<\ell\leqslant u}
            \mathds{1}_{\mathscr{B}}(\ell)e\bigg(-\frac{h\ell^\gamma}{d}\bigg)\Bigg|
            \bigg|\frac{\partial\mathfrak{f}_{h,d}(u)}{\partial u}\bigg|\mathrm{d}u\Bigg)
                \nonumber \\
 \ll & \,\, X^{\gamma-1}\sum_{d\sim D}\frac{1}{d}\sum_{0<h\leqslant H}\max_{X/2<u\leqslant X}
            \Bigg|\sum_{X/2<\ell\leqslant u}
            \mathds{1}_{\mathscr{B}}(\ell)e\bigg(-\frac{h\ell^\gamma}{d}\bigg)\Bigg|,
\end{align*}
where we use the estimate
\begin{equation*}
  \big|\mathfrak{f}_{h,d}(u)\big|\ll hd^{-1}u^{\gamma-1}\qquad \textrm{and}\qquad\,\,
  \bigg|\frac{\partial \mathfrak{f}_{h,d}(u)}{\partial u}\bigg|\ll hd^{-1}u^{\gamma-2}.
\end{equation*}
where we use the estimate
\begin{equation*}
  \big|\mathfrak{f}_h(u)\big|\ll hu^{\gamma-1}\qquad \textrm{and}\qquad\,\,
  \bigg|\frac{\partial \mathfrak{f}_h(u)}{\partial u}\bigg|\ll hu^{\gamma-2}.
\end{equation*}
There, we obtain
\begin{align}\label{error-splitting-3}
   & \,\,  \sum_{d\sim D}\Bigg|\sum_{ \ell\sim X}\mathds{1}_{\mathscr{B}}(\ell)
           \bigg(\psi\bigg(-\frac{(\ell+1)^\gamma}{d}\bigg)-\psi\bigg(-\frac{\ell^\gamma}{d}\bigg)\bigg)\Bigg|
                 \nonumber \\
 \ll & \,\, x^\gamma\mathscr{L}^{-A}+\max_{\substack{X/2<u\leqslant X\\ x^{1-\eta}<X\leqslant x}}X^{\gamma-1}
             \sum_{d\sim D}\frac{1}{d}\sum_{0<h\leqslant H}\Bigg|\sum_{X/2<\ell\leqslant u}
            \mathds{1}_{\mathscr{B}}(\ell)e\bigg(-\frac{h\ell^\gamma}{d}\bigg)\Bigg|
                      \nonumber \\
 \ll & \,\, x^\gamma\mathscr{L}^{-A}+\max_{\substack{X/2<u\leqslant X\\ x^{1-\eta}<X\leqslant x}}X^{\gamma-1}
             \sum_{d\sim D}\frac{1}{d}\sum_{0<h\leqslant H}\delta^*(d,h)\sum_{X/2<\ell\leqslant u}
            \mathds{1}_{\mathscr{B}}(\ell)e\bigg(-\frac{h\ell^\gamma}{d}\bigg)
                      \nonumber \\
 \ll & \,\, x^\gamma\mathscr{L}^{-A}+X^{\gamma-1}\max_{\substack{X/2<u\leqslant X\\ x^{1-\eta}<X\leqslant x}}
            \Bigg|\sum_{\substack{\ell\in\mathscr{B}\\ X/2<\ell\leqslant u}}\sum_{0<h\leqslant H}\Xi(h)
            e\bigg(-\frac{h\ell^\gamma}{d}\bigg)\Bigg|,
\end{align}
where
\begin{equation*}
  \Xi(h)=\sum_{d\sim D}\frac{\delta^*(d,h)}{d},\qquad \quad \big|\delta^*(d,h)\big|=1.
\end{equation*}
Denote by $\mathfrak{S}_0$ the triple summation in the absolute value on the right--hand side in (\ref{error-splitting-3}). In order to prove (\ref{error-splitting-2}), it is sufficient to show that
$\mathfrak{S}_0\ll X^{1-\eta}$. By dividing $\ell$ into two variables and following the same processes, except for replacing $1/\gamma$ with $\gamma$ and replacing the exponent pair $(\frac{11}{30},\frac{8}{15})$ with $(\frac{1}{2},\frac{1}{2})$, as demonstrated in Type II sum estimate, we deduce that (using the same notation as in (\ref{Sigma-2-fi-1}))
\begin{align}\label{S_0-upper-1}
            |\mathfrak{S}_0|^2
 \ll & \,\, X^{\eta}\big(X^{2-\gamma}TJD+XMTJD^{-1}+X^{\gamma/2+2}M^{-1}J^{5/2}D^{-1/2}T^{-1/2}
                    \nonumber \\
     & \,\, \qquad+X^{\gamma/2+1}J^{3/2}D^{-3/2}T^{1/2}\big).
\end{align}
Now, we choose
\begin{equation}\label{T-chosen-S_0}
 T=\big[X^{(\gamma+2)/3}M^{-4/3}JD^{1/3}\big].
\end{equation}
Trivially, one has $T>1$ provided that $M\ll X^{(\gamma+\xi+2)/4-\eta}$. Inserting (\ref{T-chosen-S_0}) into (\ref{S_0-upper-1}),  we obtain
\begin{align}\label{S_0-upper-2}
            |\mathfrak{S}_0|^2
 \ll & \,\, X^{\eta}\big(M^{-4/3}X^{(8-2\gamma)/3}J^2D^{4/3}+M^{-1/3}X^{(\gamma+5)/3}J^2D^{-2/3}
                     \nonumber \\
     & \,\, \qquad +M^{-2/3}X^{(2\gamma+4)/3}J^2D^{-4/3}\big),
\end{align}
which combined with $J\ll H=X^{1-\gamma+\xi+\eta}$ and $X^{\xi}\leqslant D\leqslant X^{\xi+\eta/2}$ yields
\begin{align*}
            |\mathfrak{S}_0|^2
 \ll & \,\, X^{\eta}\big(M^{-4/3}X^{(14-8\gamma+10\xi)/3}+M^{-1/3}X^{(11-5\gamma+4\xi)/3}
            +M^{-2/3}X^{(10-4\gamma+2\xi)/3}\big).
\end{align*}
Accordingly, one derives that $\mathfrak{S}_0\ll X^{1-\eta}$, provided that
\begin{equation}\label{S_0-upp-condition}
X^{5-5\gamma+4\xi+\eta}\ll M\ll X^{(\gamma+\xi+2)/4-\eta},
\end{equation}
and $\gamma>(5\xi+6)/7$. By (\ref{level-def}), we get $\mathfrak{S}_0\ll X^{1-\eta}$ under the condition (\ref{S_0-upp-condition}) and $\gamma>225/238=0.94537815\dots$. Especially, for $0.989<\gamma<1$ and the definition of $\xi$, i.e., (\ref{level-def}), it is easy to see that
\begin{equation*}
X^{5-5\gamma+4\xi+\eta}\leqslant X^{\alpha_0}<X^{\beta_0}\leqslant X^{(\gamma+\xi+2)/4-\eta},
\end{equation*}
where
\begin{equation*}
\alpha_0=0.6395925926,\qquad \beta_0=0.78378703702.
\end{equation*}
Next, we shall illustrate that, for $\ell=p_1p_2\dots p_8\in\mathscr{B}$ with $\ell\sim X>x^{1-\eta}$, there must be some partial product of $p_1p_2\dots p_8$ which lies in the interval $[X^{\alpha_0},X^{\beta_0}]$.

First, since $p_i\geqslant x^{1/17.41}$ and $p_1p_2\dots p_8\in[x^{1-\eta},x]$, we have
$p_i\leqslant X^{\beta_0}\,(i=1,2,\dots,8)$. Otherwise, if there exists some $p_i>X^{\beta_0}$, then
$\ell\geqslant X^{\beta_0}\cdot x^{7/17.41}>x^{0.783+7/17.41}>x^{1.18}$, which is a contradiction. If there exists
some $p_i\in[X^{\alpha_0},X^{\beta_0}]$, then the conclusion follows. If this case does not exist, we consider the
product $p_5p_6p_7p_8$. We claim that there must be $p_5p_6p_7p_8\leqslant X^{\beta_0}$. Otherwise, we get
\begin{equation*}
X^{\beta_0}<p_5p_6p_7p_8<X(p_1p_2p_3p_4)^{-1}<x^{1-4/17.41}<x^{0.771}<X^{\beta_0},
\end{equation*}
which is a contradiction. If $p_5p_6p_7p_8\in[X^{\alpha_0},X^{\beta_0}]$, then the conclusion follows. If this case does not exist, i.e., $p_5p_6p_7p_8<X^{\alpha_0}$, we consider the product $p_3p_4p_5p_6p_7p_8$. Now, we claim that there must hold $p_3p_4p_5p_6p_7p_8\geqslant X^{\alpha_0}$. Otherwise, we obtain
\begin{equation*}
X\ll (p_1p_2)(p_3p_4p_5p_6p_7p_8)<(p_3p_4p_5p_6p_7p_8)^{4/3}<X^{4\alpha_0/3}<X^{0.853},
\end{equation*}
which is a contradiction. If $p_3p_4p_5p_6p_7p_8\leqslant X^{\beta_0}$, then the conclusion follows. If this case does not exist, i.e., $p_3p_4p_5p_6p_7p_8> X^{\beta_0}$, we have (under the condition $p_5p_6p_7p_8<X^{\alpha_0}$)
\begin{equation}\label{1234-condition}
X^{1-\alpha_0-\eta}<(X/2)(p_5p_6p_7p_8)^{-1}<p_1p_2p_3p_4<(p_1p_2p_3p_4p_5p_6p_7p_8)^{1/2}<X^{1/2},
\end{equation}
which implies that $p_3p_4>X^{(1-\alpha_0)/2-\eta}$. By noting that
\begin{equation*}
p_3p_4<p_5p_6<(p_5p_6p_7p_8)^{1/2}<X^{\alpha_0/2},
\end{equation*}
we get $X^{(1-\alpha_0)/2-\eta}<p_3p_4<X^{\alpha_0/2}$. If $X^{1-\beta_0}\leqslant p_3p_4<X^{\alpha_0/2}$, then
we have
\begin{equation*}
X^{\alpha_0}<X^{1-\alpha_0/2-\eta}<p_1p_2p_5p_6p_7p_8\leqslant X^{\beta_0}.
\end{equation*}
If $X^{(1-\alpha_0)/2-\eta}<p_3p_4<X^{1-\beta_0}$, we have $p_3<(p_3p_4)^{1/2}<X^{(1-\beta_0)/2}$, and thus
\begin{equation*}
p_4p_5p_6p_7p_8>X^{\beta_0}p_3^{-1}>X^{\beta_0-(1-\beta_0)/2}>X^{\alpha_0}.
\end{equation*}
Moreover, there must hold $p_4p_5p_6p_7p_8\leqslant X^{\beta_0}$. Otherwise, if $p_4p_5p_6p_7p_8> X^{\beta_0}$, then
\begin{equation}\label{123-condition}
p_1p_2p_3<X(p_4p_5p_6p_7p_8)^{-1}<X^{1-\beta_0}.
\end{equation}
From (\ref{1234-condition}) and (\ref{123-condition}), we deduce that
\begin{equation*}
X^{\beta_0-\alpha_0-\eta}<X^{1-\alpha_0-\eta}(p_3p_4)^{-1}<p_1p_2<(p_1p_2p_3)^{2/3}<X^{2(1-\beta_0)/3},
\end{equation*}
which is contradict to the fact that
\begin{equation*}
\beta_0-\alpha_0-\eta>\frac{2}{3}(1-\beta_0).
\end{equation*}
This completes the proof of Lemma \ref{omega=8-error}.
\end{proof}

From Lemma \ref{upper-lower-sieve} and Lemma \ref{omega=8-error}, we deduce that
\begin{equation}\label{S(E)-upper}
S\big(\mathscr{E},x^{\gamma/2}\big)\leqslant \mathcal{X}V\big(x^{\gamma/2}\big)\bigg(F\bigg(\frac{2\xi}{\gamma}\bigg)+o(1)\bigg).
\end{equation}
According to (\ref{V(z)-def}), we get
\begin{equation*}
V\big(x^{\gamma/2}\big)=\frac{2}{17.41\gamma}V\big(x^{1/17.41}\big)\big(1+O(\log x)^{-1}\big),
\end{equation*}
which combined with (\ref{S(E)-upper}) yields
\begin{equation}\label{S(E)-upper-1}
S(\mathscr{E},x^{\gamma/2})\leqslant \frac{2e^{C_0}}{17.41\xi}\mathcal{X}V\big(x^{1/17.41}\big)(1+o(1)).
\end{equation}
Next, we compute the quantity $\mathcal{X}$ definitely. Obviously, we have
\begin{equation}\label{X-expan}
\mathcal{X}=\sum_{\ell\in\mathscr{B}}\gamma\ell^{\gamma-1}+
\sum_{\ell\in\mathscr{B}}\big((\ell+1)^\gamma-\ell^\gamma-\gamma\ell^{\gamma-1}\big).
\end{equation}
For the second term in (\ref{X-expan}), we have
\begin{align}\label{X-expan-nd-upper}
  & \,\,   \sum_{\ell\in\mathscr{B}}\big((\ell+1)^\gamma-\ell^\gamma-\gamma\ell^{\gamma-1}\big)
        \ll \sum_{\ell\in\mathscr{B}}\ell^{\gamma-2}
              \nonumber \\
 \ll& \,\, \Bigg(\sum_{x^{1/17.41}\leqslant p\leqslant x}p^{\gamma-2}\Bigg)^8
 \ll  \Bigg(\sum_{x^{1/17.41}\leqslant m\leqslant x}m^{\gamma-2}\Bigg)^8\ll x^{8(\gamma-1)/17.41}=o(1).
\end{align}
For the first term in (\ref{X-expan}), we have
\begin{align}\label{X-first-num}
   & \,\, \sum_{\ell\in\mathscr{B}}\gamma\ell^{\gamma-1}
                \nonumber \\
 = & \,\, \gamma\sum_{x^{1/17.41}\leqslant p_1< x^{1/8}}\sum_{p_1<p_2<(x/p_1)^{1/7}}
          \sum_{p_2<p_3<(x/(p_1p_2))^{1/6}}\sum_{p_3<p_4<x/(p_1p_2p_3)^{1/5}}
          \sum_{p_4<p_5<x/(p_1p_2p_3p_4)^{1/4}}
                \nonumber \\
    & \,\,\times\sum_{p_5<p_6<x/(p_1p_2p_3p_4p_5)^{1/3}}
          \sum_{p_6<p_7<x/(p_1p_2p_3p_4p_5p_6)^{1/2}}
          \sum_{p_7<p_8<x/(p_1p_2p_3p_4p_5p_6p_7)}
          (p_1p_2p_3p_4p_5p_6p_7p_8)^{\gamma-1}
                \nonumber \\
 = & \,\, \gamma\big(1+o(1)\big)\int_{x^{1/17.41}}^{x^{1/8}}\int_{u_1}^{(\frac{x}{u_1})^{1/7}}
          \int_{u_2}^{(\frac{x}{u_1u_2})^{1/6}}\int_{u_3}^{(\frac{x}{u_1u_2u_3})^{1/5}}
          \int_{u_4}^{(\frac{x}{u_1u_2u_3u_4})^{1/4}}\int_{u_5}^{(\frac{x}{u_1u_2u_3u_4u_5})^{1/3}}
                \nonumber \\
   & \,\, \times\int_{u_6}^{(\frac{x}{u_1u_2u_3u_4u_5u_6})^{1/2}}
          \int_{u_7}^{\frac{x}{u_1u_2u_3u_4u_5u_6u_7}}
          \frac{(u_1u_2\dots u_8)^{\gamma-1}\mathrm{d}u_8\mathrm{d}u_7\dots\mathrm{d}u_2\mathrm{d}u_1}
          {(\log u_1)(\log u_2)\dots(\log u_8)}
                 \nonumber \\
 = & \,\, \gamma\big(1+o(1)\big)\int_{\frac{1}{17.41}}^{\frac{1}{8}}\frac{\mathrm{d}t_1}{t_1}
          \int_{t_1}^{\frac{1-t_1}{7}}\frac{\mathrm{d}t_2}{t_2}
          \int_{t_2}^{\frac{1-t_1-t_2}{6}}\frac{\mathrm{d}t_3}{t_3}
          \int_{t_3}^{\frac{1-t_1-t_2-t_3}{5}}\frac{\mathrm{d}t_4}{t_4}
          \int_{t_4}^{\frac{1-t_1-t_2-t_3-t_4}{4}}\frac{\mathrm{d}t_5}{t_5}
                 \nonumber \\
   & \,\, \times\int_{t_5}^{\frac{1-t_1-t_2-t_3-t_4-t_5}{3}}\frac{\mathrm{d}t_6}{t_6}
          \int_{t_6}^{\frac{1-t_1-t_2-t_3-t_4-t_5-t_6}{2}}\frac{\mathrm{d}t_7}{t_7}
          \int_{t_7}^{1-t_1-t_2-t_3-t_4-t_5-t_6-t_7}
          \frac{x^{(t_1+t_2+\dots+t_8)\gamma}}{t_8}\mathrm{d}t_8.
\end{align}
For the innermost integral in (\ref{X-first-num}), we have
\begin{align}\label{inner-expli}
     & \,\,  \int_{t_7}^{1-t_1-t_2-t_3-t_4-t_5-t_6-t_7}\frac{x^{(t_1+t_2+\dots+t_8)\gamma}}{t_8}\mathrm{d}t_8
                  \nonumber \\
   = & \,\, \frac{1}{\gamma \log x}\int_{t_7}^{1-t_1-t_2-t_3-t_4-t_5-t_6-t_7}\frac{1}{t_8}\mathrm{d}x^{(t_1+t_2+\dots+t_8)\gamma}
                  \nonumber \\
   = & \,\, \frac{1}{\gamma \log x}\bigg(\frac{x^\gamma}{1-t_1-t_2-t_3-t_4-t_5-t_6-t_7}
            +O\bigg(\frac{x^\gamma}{\log x}\bigg)\bigg)
                   \nonumber \\
   = & \,\, \frac{1}{1-t_1-t_2-t_3-t_4-t_5-t_6-t_7}\cdot\frac{x^\gamma}{\gamma\log x}(1+o(1)).
\end{align}
From (\ref{X-expan-nd-upper})--(\ref{inner-expli}), we deduce that
\begin{align}\label{X-compu-num}
          \mathcal{X}
 = & \,\, \frac{x^\gamma(1+o(1))}{\log x}\int_{\frac{1}{17.41}}^{\frac{1}{8}}\frac{\mathrm{d}t_1}{t_1}
          \int_{t_1}^{\frac{1-t_1}{7}}\frac{\mathrm{d}t_2}{t_2}
          \int_{t_2}^{\frac{1-t_1-t_2}{6}}\frac{\mathrm{d}t_3}{t_3}
          \int_{t_3}^{\frac{1-t_1-t_2-t_3}{5}}\frac{\mathrm{d}t_4}{t_4}
          \int_{t_4}^{\frac{1-t_1-t_2-t_3-t_4}{4}}\frac{\mathrm{d}t_5}{t_5}
                   \nonumber \\
   & \,\, \times\int_{t_5}^{\frac{1-t_1-t_2-t_3-t_4-t_5}{3}}\frac{\mathrm{d}t_6}{t_6}
          \int_{t_6}^{\frac{1-t_1-t_2-t_3-t_4-t_5-t_6}{2}}
          \frac{\mathrm{d}t_7}{t_7(1-t_1-t_2-t_3-t_4-t_5-t_6-t_7)}.
\end{align}
Combining (\ref{omega(a)<7-lower}), (\ref{W(a,1/17.41)-lower}), (\ref{E-trans-upper}), (\ref{S(E)-upper-1}) and (\ref{X-compu-num}), we obtain
\begin{align*}
                  \sum_{\substack{a\in\mathscr{A}\\(a,P(x^{1/17.41}))=1\\ \Omega(a)\leqslant7}}\mathscr{W}_a
 \geqslant & \,\,  \frac{2e^{C_0}}{17.41}\pi(x^\gamma)V\big(x^{1/17.41}\big)(1+o(1))
                  \Bigg(\frac{\log\big(17.41\xi-1\big)}{\xi}-
                  \lambda\int_u^{17.41}\frac{t-u}{t(t\xi-1)}\mathrm{d}t
                            \nonumber \\
 & \,\, -\frac{\lambda\gamma}{\xi}\int_{\frac{1}{17.41}}^{\frac{1}{8}}\frac{\mathrm{d}t_1}{t_1}
          \int_{t_1}^{\frac{1-t_1}{7}}\frac{\mathrm{d}t_2}{t_2}
          \int_{t_2}^{\frac{1-t_1-t_2}{6}}\frac{\mathrm{d}t_3}{t_3}
          \int_{t_3}^{\frac{1-t_1-t_2-t_3}{5}}\frac{\mathrm{d}t_4}{t_4}
          \int_{t_4}^{\frac{1-t_1-t_2-t_3-t_4}{4}}\frac{\mathrm{d}t_5}{t_5}
                   \nonumber \\
   & \,\, \times\int_{t_5}^{\frac{1-t_1-t_2-t_3-t_4-t_5}{3}}\frac{\mathrm{d}t_6}{t_6}
          \int_{t_6}^{\frac{1-t_1-t_2-t_3-t_4-t_5-t_6}{2}}
          \frac{\mathrm{d}t_7}{t_7(1-\sum_{j=1}^7t_j)}\Bigg)+O\big(x^{\frac{16.41}{17.41}+\varepsilon}\big).
\end{align*}
By simple numerical calculations, it is easy to see that the number in the above brackets $(\cdot)$ is $\geqslant0.00024867$, provided that $0.989<\gamma<1$. This completes the proof of Theorem \ref{Theorem-1}.

\section*{Acknowledgement}

%\noindent
The authors would like to appreciate the referee for his/her patience in refereeing this paper.
This work is supported by Beijing Natural Science Foundation (Grant No. 1242003), and
the National Natural Science Foundation of China (Grant Nos. 11901566, 12001047, 11971476, 12071238).
Especially, Fei Xue is supported by the China Scholarship Council (Grant No. 202306430075).

\end{document}